\documentclass{amsproc}
\usepackage{amssymb,amsmath,amsthm,amsfonts,mathrsfs}
\usepackage[hmargin=3cm,vmargin=3.5cm]{geometry}
\usepackage[dvipsnames,table,xcdraw]{xcolor}
\usepackage[colorlinks = true,
            linkcolor = Plum,
            urlcolor  = Blue,
            citecolor = ForestGreen,
            anchorcolor = blue]{hyperref}
\usepackage{graphicx,psfrag}
\usepackage{amscd}  
\usepackage[shellescape]{gmp}
\usepackage{stmaryrd}  
\usepackage[all,2cell]{xy} \UseAllTwocells \SilentMatrices
\usepackage{tikz-cd}
\usepackage[dvips]{epsfig}
\usepackage{MnSymbol}
\usepackage{comment}
\usepackage{enumitem}
\usepackage{kbordermatrix}

\usetikzlibrary{arrows,automata}
\usetikzlibrary{decorations.markings}
\usetikzlibrary{patterns}

\pgfdeclarepatternformonly{mydots}{\pgfqpoint{-1pt}{-1pt}}{\pgfqpoint{1pt}{1pt}}{\pgfqpoint{2pt}{2pt}}
{%
  \pgfpathcircle{\pgfqpoint{0pt}{0pt}}{.5pt}%
  \pgfusepath{fill}%
}%

\pgfdeclarepatternformonly{mynewdots}{\pgfqpoint{-1pt}{-1pt}}{\pgfqpoint{1pt}{1pt}}{\pgfqpoint{1pt}{1pt}}
{%
  \pgfpathcircle{\pgfqpoint{0pt}{0pt}}{.5pt}%
  \pgfusepath{fill}%
}%

\usetikzlibrary{decorations.pathmorphing}
\tikzset{snake it/.style={decorate, decoration=snake}}

\makeatletter
\def\@adminfootnotes{%
  \let\@makefnmark\relax  \let\@thefnmark\relax
  \ifx\@empty\@date\else \@footnotetext{\@setdate}\fi
  \ifx\@empty\@subjclass\else \@footnotetext{\@setsubjclass}\fi
  \ifx\@empty\@keywords\else \@footnotetext{\@setkeywords}\fi
  \ifx\@empty\thankses\else \@footnotetext{%
    \def\par{\let\par\@par}\@setthanks}%
  \fi
}
\makeatother

\theoremstyle{definition}
\newtheorem{theorem}{Theorem}[section]

\newtheorem{remark}[theorem]{Remark}

\newtheorem{prop}[theorem]{Proposition}

\newtheorem{corollary}[theorem]{Corollary}

\newtheorem{example}[theorem]{Example}

\let\oldtocsection=\tocsection
\let\oldtocsubsection=\tocsubsection
\renewcommand{\tocsection}[2]{\hspace{0em}\oldtocsection{#1}{#2}}
\renewcommand{\tocsubsection}[2]{\hspace{1em}\oldtocsubsection{#1}{#2}}

\usepackage{chngcntr}
\counterwithin{figure}{subsection}

\makeatletter
\@namedef{subjclassname@2020}{%
  \textup{2020} Mathematics Subject Classification}

\makeatother

\title[One-dimensional topological theories with defects]{One-dimensional topological theories with defects: the linear case}

\author{Mee Seong Im}
 \address{Department of Mathematics, United States Naval Academy, Annapolis, MD 21402, USA}
 \email{\href{mailto:meeseongim@gmail.com}{meeseongim@gmail.com}}
\author{Mikhail Khovanov} 
 \address{Department of Mathematics, Columbia University, New York, NY 10027, USA}
 \email{\href{mailto:khovanov@math.columbia.edu}{khovanov@math.columbia.edu}}
 
\date{April 4, 2023}

\providecommand{\keywords}[1]{\textbf{\textit{Key words and phrases.}} #1}

\begin{document}

\def\gen{\mathsf{generators}}
\def\im{\mathsf{im}}
\def\init{\mathsf{in}}
\def\t{\mathsf{t}}
\def\out{\mathsf{out}}
\def\I{\mathsf I}
\def\R{\mathbb R}
\def\Q{\mathbb Q}
\def\Z{\mathbb Z}
\def\mc{\mathcal{c}}
\def\Kar{\mathsf{Kar}}

\def\mcI{\mathcal{I}}
\def\mmcI{\mathcal{I}} 
\def\mcK{\mathcal{K}}
\def\mmcK{\mathcal{K}} 
\def\mcF{\mathcal{F}} 

\def\mcU{\mathcal{U}}
\def\N{\mathbb N} 
\def\C{\mathbb C}
\def\S{\mathbb S}
\def\SS{\mathbb S} 
\def\CP{\mathbb P}
\def\Ob{\mathsf{Ob}}
\def\new{\mathsf{new}}
\def\old{\mathsf{old}}
\def\op{\mathsf{op}}
\def\rat{\mathsf{rat}}
\def\rec{\mathsf{rec}}
\def\coev{\mathsf{coev}}
\def\ev{\mathsf{ev}}
\def\id{\mathsf{id}}
\def\Notation{\textsf{Notation}}
\def\circleft{\raisebox{-.18ex}{\scalebox{1}[2.25]{\rotatebox[origin=c]{180}{$\curvearrowright$}}}}
\renewcommand\SS{\ensuremath{\mathbb{S}}}
\newcommand{\kllS}{\kk\llangle  S \rrangle} 
\newcommand{\kllSS}[1]{\kk\llangle  #1 \rrangle}
\newcommand{\klS}{\kk\langle S\rangle}  
\newcommand{\aver}{\mathsf{av}}  
\newcommand{\ophana}{\overline{\phantom{a}}}
\newcommand{\Bool}{\mathbb{B}}
\newcommand{\dmod}{\mathsf{-mod}}
\newcommand{\pfmod}{\mathsf{-pfmod}}
\newcommand{\primitive}{\mathsf{irr}}
\newcommand{\Bmod}{\Bool\mathsf{-mod}}  
\newcommand{\Bmodo}[1]{\Bool_{#1}\mathsf{-mod}}  
\newcommand{\Bfmod}{\Bool\mathsf{-fmod}} 
\newcommand{\Bfpmod}{\Bool\mathsf{-fpmod}} 
\newcommand{\Bfsmod}{\Bool\mathsf{-}\underline{\mathsf{fmod}}}  
\newcommand{\undvar}{\underline{\varepsilon}} 
\newcommand{\undDC}{\underline{\mathcal{DC}}}
\newcommand{\KarC}{\mathsf{Kar}(\mcC)}  
\newcommand{\KarCa}{\mathsf{Kar}(\mcC_{\alpha})} 
\newcommand{\undotimes}{\underline{\otimes}}
\newcommand{\sigmaacirc}{\Sigma^{\ast}_{\circ}} 
\newcommand{\cl}{\mathsf{cl}}
\newcommand{\PP}{\mathcal{P}} 
\newcommand{\wedgezero}{\{ \vee ,0\} } 
\newcommand{\whA}{\widehat{A}}
\newcommand{\whC}{\widehat{C}}
\newcommand{\whM}{\widehat{M}}
\newcommand{\Sigmalr}{\Sigma^{\Z}}
\newcommand{\Sigmal}{\Sigma^{-}}
\newcommand{\Sigmar}{\Sigma^{+}}
\newcommand{\Sigmaa}{\Sigma^{\ast}}
\newcommand{\SigmaZ}{\Sigma^{\Z}}  
\newcommand{\Sigmac}{\Sigma^{\circ}}
\newcommand{\Span}{\mathsf{Span}}
\newcommand{\oneb}{\mathbf{1}}
\newcommand{\wmcC}{\widetilde{\mcC}}
\newcommand{\wmcCa}{\wmcC_{\alpha}}
\newcommand{\omcC}{\overline{\mcC}}

\newcommand{\alphalr}{\alpha_{\leftrightarrow}}
\newcommand{\alphaZ}{\alpha_{\Z}}
\newcommand{\mcCinfalpha}{\mcC^{\infty}_{\alpha}}
\newcommand{\mcCp}{\mathcal{C}'}
\newcommand{\kkvect}{\kk\mathsf{-vect}}
\newcommand{\Mat}{\mathsf{Mat}}

\let\oldemptyset\emptyset
\let\emptyset\varnothing

\newcommand{\undempty}{\underline{\emptyset}}
\def\basis{\mathsf{basis}}
\def\irr{\mathsf{irr}} 
\def\spanning{\mathsf{spanning}}
\def\elmt{\mathsf{elmt}}

\def\l{\lbrace}
\def\r{\rbrace}
\def\o{\otimes}
\def\lra{\longrightarrow}
\def\hooklra{\raisebox{.2ex}{$\subset$}\!\!\!\raisebox{-0.21ex}{$\longrightarrow$}}
\def\circleft{\raisebox{-.18ex}{\scalebox{1}[2.25]{\rotatebox[origin=c]{180}{$\curvearrowright$}}}}
\def\Ext{\mathsf{Ext}}
\def\mf{\mathfrak} 
\def\mcC{\mathcal{C}}
\def\mcO{\mathcal{O}}
\def\Fr{\mathsf{Fr}}

\def\ovb{\overline{b}}
\def\tr{{\sf tr}} 
\def\det{{\sf det }} 
\def\tral{\tr_{\alpha}}
\def\one{\mathbf{1}}   

\def\lra{\longrightarrow}
\def\twoheadlra{\longrightarrow\hspace{-4.6mm}\longrightarrow}
\def\hooklra{\raisebox{.2ex}{$\subset$}\!\!\!\raisebox{-0.21ex}{$\longrightarrow$}}
\def\kk{\mathbf{k}}  
\def\gdim{\mathsf{gdim}}  
\def\rk{\mathsf{rk}}
\def\undep{\underline{\epsilon}}
\def\mathM{\mathbf{M}}  

\def\CCC{\mathcal{C}} 
\def\wCCC{\widehat{\CCC}}  

\def\complement{\mathsf{comp}}
\def\Rec{\mathsf{Rec}} 

\def\Cob{\mathsf{Cob}} 
\def\TCob{\mathsf{TCob}}
\def\TCobt{\TCob_2}  
\def\OCCobt{\mathsf{OCCob_2}} 
\def\Kar{\mathsf{Kar}}   

\def\dmod{\mathsf{-mod}}   
\def\pmod{\mathsf{-pmod}}    

\newcommand{\alphai}{\alpha_{\I}}  
\newcommand{\alphac}{\alpha_{\circ}}  
\newcommand{\alphap}{(\alphai,\alphac)} 

\newcommand{\brak}[1]{\ensuremath{\left\langle #1\right\rangle}}
\newcommand{\oplusop}[1]{{\mathop{\oplus}\limits_{#1}}}
\newcommand{\ang}[1]{\langle #1 \rangle } 
\newcommand{\ppartial}[1]{\frac{\partial}{\partial #1}} 

\newcommand{\mcA}{{\mathcal A}}
\newcommand{\cZ}{{\mathcal Z}}
\newcommand{\sq}{$\square$}
\newcommand{\bi}{\bar \imath}
\newcommand{\bj}{\bar \jmath}

\newcommand{\undn}{\mathbf{n}}
\newcommand{\undm}{\mathbf{m}}
\newcommand{\cob}{\mathsf{cob}} 
\newcommand{\comp}{\mathsf{comp}} 

\newcommand{\Aut}{\mathsf{Aut}}
\newcommand{\Hom}{\mathsf{Hom}}
\newcommand{\Ind}{\mbox{Ind}}
\newcommand{\Id}{\textsf{Id}}
\newcommand{\End}{\mathsf{End}}
\newcommand{\iHom}{\underline{\mathsf{Hom}}}
\newcommand{\Bools}{\Bool^{\mathfrak{s}}}
\newcommand{\mfs}{\mathfrak{s}}

\newcommand{\drawing}[1]{
\begin{center}{\psfig{figure=fig/#1}}\end{center}}

\def\endomCempt{\End_{\mcC}(\emptyset_{n-1})}

\def\endomCempt{\End_{\mcC}(\emptyset_{n-1})}

\def\MS#1{{\color{blue}[MS: #1]}}
\def\MK#1{{\color{red}[MK: #1]}}
\def\bfred#1{{\bf \color{red}{#1}}}\begin{abstract} The paper studies the Karoubi envelope of a one-dimensional topological theory with defects and inner endpoints, defined over a field. It turns out that the Karoubi envelope is determined by a symmetric Frobenius algebra $\mcK$ associated to the theory. The Karoubi envelope is then equivalent to the quotient of the Frobenius--Brauer category of $\mcK$ modulo the ideal of  negligible morphisms. Symmetric Frobenius algebras, such as $\mcK$, describe two-dimensional TQFTs for the category of thin flat surfaces, and elements of the algebra can be turned into defects on the side boundaries of these surfaces. We also explain how to couple $\mcK$ to the universal construction restricted to closed surfaces to define a topological theory of open-closed two-dimensional cobordisms which is usually not an open-closed 2D TQFT.  
\end{abstract}

\maketitle
\tableofcontents

%
%

\section{Introduction}
\label{section:intro}
Universal construction~\cite{BHMV,Kh3} starts with an evaluation function for closed $n$-manifolds to produce state space for closed $(n-1)$-manifolds and maps between these spaces associated to $n$-cobordisms. This results in a functor from the category of $n$-dimensional cobordisms to the category of vector spaces (if the evaluation function takes values in a field) which usually fails to be a TQFT, with the tensor product of state spaces for two $(n-1)$-manifolds $N_1,N_2$ properly embedded into the state space for their union:
\begin{equation}
     A(N_1)\otimes A(N_2) \ \hooklra \ A(N_1\sqcup N_2). 
\end{equation}
Universal construction for foams in $\R^3$ in place of $n$-cobordisms is used as an intermediate step in constructing link homology theories~\cite{Kh1,MV,RW1}, see also a review in~\cite{KK} and papers~\cite{Kh2,Me} for other uses and references for the universal construction. 

The universal construction turns out to be interesting already in low dimensions, including in dimensions two~\cite{Kh2,KS2,KKO,KQR} and  one~\cite{Kh3,IK-automata,IZ,IKV23,GIKKL23}. In the latter case, one needs to add zero-dimensional defects ($0$-submanifolds) with labels in a set $\Sigma$. An oriented interval with a collection of $\Sigma$-labelled defects encodes a word $\omega$, that is, an element of the free monoid $\Sigma^{\ast}$ on the set $\Sigma$. An oriented circle with labels in $\Sigma$ encodes a word up to cyclic equivalence. Given an evaluation of each word and a separate evaluation of words up to  cyclic equivalence, there is an associated rigid linear monoidal category, defined in studied in~\cite{Kh3}. It is straightforward to see~\cite{Kh3} that the hom spaces in the resulting categories are finite-dimensional if and only if the evaluations are given by \emph{rational} noncommutative power series~\cite{BR1,RRV}. 

\vspace{0.1in} 

In the present paper we study this category for a rational evaluation $\alpha$. The Karoubi closure of the resulting category can be reduced to the Karoubi closure of a category built from a symmetric Frobenius algebra $\mcK$ that can be extracted from $\alpha$, as explained in Section~\ref{subsec_karoubi_env}.  Sections~\ref{subsection_oned_defectTQFT}-\ref{subsec_top_theory} are devoted to the setup, basic theory and various examples. In Section~\ref{section_2d} we review thin flat surface 2D TQFTs associated to symmetric Frobenius algebras and explain how to enhance these TQFTs by $0$-dimensional defects floating along the boundary that carry elements of the algebra. Throughout the paper we run comparisons between one-dimensional theories with defects and two-dimensional theories without defects and discuss nonsemisimple versus semisimple TQFTs in two dimensions. 

\vspace{0.1in} 

The Boolean analogues of these categories and their relation to automata and regular languages are investigated in~\cite{IK-automata}, where the absence of linear structure creates additional complexities. 

\vspace{0.1in} 

{\bf Acknowledgments.}
We would like to thank Aaron Lauda and  Vladimir Retakh for illuminating discussions. M.S.I. was partially supported by Naval Academy Research Council (Jr. NARC) Fellowship over the summer.  
M.K. gratefully acknowledges partial support via NSF grants DMS-1807425, DMS-2204033 and Simons Collaboration Award 994328.


\section{One-dimensional topological theories with defects over a field} 
\label{section:encoding-1-dim-top-theory}

We fix a ground field $\kk$. 


\subsection{A one-dimensional defect TQFT from a noncommutative power series}
\label{subsection_oned_defectTQFT}
Start with a finite set (or alphabet) $\Sigma$. Let $\Sigma^{\ast}$ be the set of finite words in letters of $\Sigma$ (elements of $\Sigma$), and $\Sigma^{\ast}_{\circ}$ be the set of circular words, \textit{i.e.}, elements of $\Sigma^{\ast}$ up to the equivalence relation $\omega_1\omega_2\sim\omega_2\omega_1$. The empty word $\emptyset$ is included in both $\Sigma^{\ast}$ and $\Sigma^{\ast}_{\circ}$. 
Suppose we are given an evaluation 
\begin{equation}\label{eq_eva}
    \alpha=(\alphai,\alphac),
\end{equation} 
where 
\begin{equation}\label{eq_evb}
    \alphai:\Sigma^{\ast}\lra \kk, \ \ \  \alphac:\Sigma^{\ast}_{\circ}\lra \kk
\end{equation} 
are two functions on words and circular words in $\Sigma$, respectively, with values in a field $\kk$. To $\alpha$, following~\cite{Kh3} (also see earlier work~\cite{KS2,KKO} for a similar framework), there is assigned a symmetric $\kk$-linear monoidal category $\mcC_{\alpha}$. Its objects are finite sign sequences $\varepsilon$, thought of as oriented $0$-manifolds, and morphisms are $\kk$-linear combinations of oriented 1-cobordisms with 0-dimensional defects, the latter decorated by elements of $\Sigma$. See \ref{s2.001}. One-cobordisms can have ``inner'' boundary points, in addition to ``outer'' boundary points that define the objects for the morphism. Forming the composition of two such cobordisms may result in components without any ``outer'' boundary points, called \emph{floating} intervals and circles. See \ref{s2.002}. These floating components are evaluated via the interval and circle evaluation functions $\alphai$ and $\alphac$, respectively, and the composition is then reduced to a diagram without floating connected components.

\input{s2.001}

\input{s2.002}

Two linear combinations of such morphisms between sequences $\varepsilon$ and $\varepsilon'$ are equal if any closures of these two linear combinations evaluate to the same element of $\kk$ via $\alpha$, see~\cite{KS2,Kh3}. 

The state space $A_{\alpha}(\varepsilon):=\Hom_{\mcC_{\alpha}}(\emptyset_0,\varepsilon)$ of a sequence $\varepsilon$ is defined as the space of homs from the empty sign sequence $\emptyset_0$ to $\varepsilon$ in $\mcC_{\alpha}$ (we use different notations $\emptyset\in \kk\Sigma^{\ast}$ for the empty word and $\emptyset_0$ for the empty oriented $0$-manifold and the corresponding object of $\mcC_{\alpha}$). 

\vspace{0.1in} 

We say that $\alpha$ is \emph{rational} if the following equivalent conditions hold: 
\begin{itemize}
    \item State spaces $A_{\alpha}(\varepsilon)$ are finite-dimensional for all $\varepsilon$.
    \item  Hom spaces in $\mcC_{\alpha}$ are finite-dimensional.
    \item Spaces $A(+)$ and $A(+-)$ are finite-dimensional. 
    \item $A(+-)$ is finite-dimensional. 
\end{itemize}
Each of these conditions is equivalent to both of the noncommutative power series $\alphai,\alphac$ in \eqref{eq_eva}, \eqref{eq_evb} being rational in the sense of~\cite{BR1} (having a finite-dimensional state space, equivalently, a finite-dimensional syntactic algebra, see also~\cite{Kh3}).  

Category $\mcC_{\alpha}$ is $\kk$-linear and preadditive.
For rational $\alpha$, it is convenient to consider the Karoubi closure $\KarCa$  of $\mcC_{\alpha}$ (also denoted $\undDC_{\alpha}$ in~\cite{KS2,Kh3} in this and related cases)  given by forming finite direct sums of objects and then adding objects for idempotent endomorphisms of these direct sums. The category $\KarCa$ is $\kk$-linear, additive, idempotent-closed, with finite-dimensional hom spaces over $\kk$. There is a fully-faithful functor \begin{equation}
    \mcC_{\alpha}\lra \KarCa. 
\end{equation}


Assume from now on that evaluation $\alpha$ is rational. 
$\kk$-vector space $A(+)$ is a left $\kk\Sigma^{\ast}$-module, that is, a module over the ring of noncommutative polynomials in letters in $\Sigma$, equivalently the monoid algebra of the free monoid $\Sigma^{\ast}$. Module $A(+)$ has a distinguished element $|\emptyset\rangle$ corresponding to the diagram with an empty word, and action of $\omega\in \Sigma^{\ast}$ that takes it to $|\omega\rangle$. Element $|\emptyset\rangle$ is a cyclic vector in $A(+)$, so  $A(+)=\kk\Sigma^{\ast}|\emptyset\rangle$; see Figure~\ref{linear-0001}, top row.

\input{linear-0001}

 $A(-)$ is the dual vector space of $A(+)$, also carrying a distinguished vector $\langle \emptyset |$, with a right action of $\kk\Sigma^{\ast}$; see Figure~\ref{linear-0001}, bottom row.

\vspace{0.1in}

The space $A(+)$ comes with the trace map 
\begin{equation}
\tr \ : \ A(+)\lra \kk, \ \ \     \tr(|\omega\rangle) \ = \  \alpha_\I(\omega).
\end{equation}
Diagrammatically, we evaluate an oriented interval with the word $\omega$ written on it using $\alpha_\I$. The trace map is nondegenerate: if $x\in A(+)$ with $x\not=0$, then there exists $\omega\in \Sigma^{\ast}$ such that $\tr(\omega x)\not= 0$. 

The trace map is part of the perfect pairing 
\begin{equation}\label{eq_pair} 
   (\:\:,\:\:) \ :\  A(-) \otimes A(+)  \ \lra  \kk, \mbox{ where }  \langle \omega_1|\otimes |\omega_2 \rangle \ \mapsto \ \alpha_\I(\omega_1\omega_2), 
\end{equation}
given by concatenating words $\omega_1,\omega_2$ written on  ``half-intervals" into the word $\omega_1\omega_2$ on an interval and evaluating it, see Figure~\ref{linear-0002}. 


\input{linear-0002}


The pairing makes the left action of $\kk\Sigma^{\ast}$ on $A(+)$ and the right action of $\kk\Sigma^{\ast}$ on $A(-)$ adjoint: 
\[  (x\omega , y ) = (x,\omega y) = \alpha_\I(x\omega y) , \ \  x\in A(-),\ \ y\in A(+), \ \ \omega\in \Sigma^{\ast}. 
\] 

Action of $\Sigma^{\ast}$ on $A(+)$ induces a $\kk$-algebra homomorphism 
\begin{equation}
    \kk \Sigma^{\ast} \stackrel{\phi_{\alpha}}{\lra} \End_{\kk}(A(+))
\end{equation}
of the algebra of noncommutative polynomials into a finite-dimensional matrix algebra. 
Denote by 
\begin{equation} 
B_0 \ := \ \im(\phi_{\alpha})
\end{equation} 
the image of $\kk\Sigma^{\ast}$ in $\End_{\kk}(A(+))$. It is a unital subalgebra of 
the matrix algebra. The algebra $B_0$ is the entire $\End_{\kk}(A(+))$ if and only if the representation $A(+)$ of $\kk\Sigma^{\ast}$ is absolutely irreducible. Note that $A(+),A(-),B_0$ and the actions above depend only on $\alpha_{\I},$ not on $\alpha_{\circ}$. 

\vspace{0.1in} 

Vice versa, suppose we are given a finite-dimensional representation $V$ of $\kk\Sigma^{\ast}$ with a cyclic vector $v^0$  and a nondegenerate trace $\tr:V\lra \kk$. Here $v^0$ corresponds to the undecorated upward-oriented half-interval (half-interval with the empty word $\emptyset$ on it). The trace is nondegenerate in the sense that for any $v\in V, v\not= 0$ there exists a word $\omega$ such that $\tr(\omega v)\not=0$. 

To this data one assigns rational noncommutative series via the  evaluation $\alpha_\I(\omega)=\tr(\omega v^0)$ for $\omega\in\Sigma^{\ast}$ so that $A(+)=V=\kk\Sigma^{\ast}v^0$ and $A(-)=V^{\ast}$, where $v^0$ is a cyclic vector.  This gives a bijection between isomorphism classes of rational evaluations (rational noncommutative power series in $\Sigma$) and nondegenerate triples $(V,v^0,\tr)$ with an action of $\Sigma^{\ast}$. The action of infinite-dimensional algebra $\kk\Sigma^{\ast}$ on $V$ factors through a faithful action of the finite-dimensional algebra $B_0\subset \End_{\kk}(V)$. 

Also, given a finite-dimensional algebra $B_0$ with a set of generators $\Sigma$, its faithful action on a finite-dimensional vector space $V$ with a cyclic vector $v^0$ and a nondegenerate trace form on $V$, this recovers the noncommutative series $\alpha_\I$ via the above recipe, with $A(+)\cong V$.

\vspace{0.1in} 

There is a minimalist way to extend the above structure of $A(+)$ with an action of $\Sigma^{\ast}$, a cyclic vector and a trace to a symmetric monoidal category $\mcCp_{\alphai}$, which turns out to be a TQFT with defects. In this construction evaluation of decorated circles is derived from that for decorated intervals. To define $\mcCp_{\alphai}$, first enhance the graphical calculus for decorated half-intervals by picking a basis $\{v^1,\dots, v^k\}$ of $A(+)$ and the dual basis $\{v_1,\dots, v_k\}$ of $A(-)$, with $k=\dim A(-)=\dim A(+)$. Denote vectors $v^i,v_i$ by placing the label $i$ at the end of the suitably oriented half-interval, see Figure~\ref{linear-0003}, left.  A half-interval decorated by $\omega\in \Sigma^{\ast}$ can be written as a linear combination of the basis vectors, by writing $\langle \omega |$ and $|\omega \rangle$ in these bases, see Figure~\ref{linear-0003}, middle and right. In the special case of undecorated half-intervals, $\omega=\emptyset$, and 
\begin{equation}\label{eq_emptyset}
    \langle \emptyset | \ = \  \sum_{i=1}^k \lambda_i^{\emptyset}v_i , \ \ \ 
    |\emptyset \rangle \ = \  \sum_{i=1}^k
    \lambda_i^{\emptyset} v^i, \ \ \ \lambda_i^{\emptyset}\in \kk.  
\end{equation}

\vspace{0.1in} 
 
\input{linear-0003}

There is then the surgery formula (the dual basis relation), shown in Figure~\ref{linear-0004} left, for cutting any half-interval in the middle, where  any floating intervals that appear are evaluated via $\alphai$. This formula also tells us that, for consistency, a circle carrying word $\omega$ should be evaluated to the trace of $\omega$ acting on $A(+)$ (equivalently, on $A(-)$), 
\begin{equation}\label{eq_trace_omega}
    \alpha_{\circ}(\omega) \ := \ \tr_{A(+)}(\omega) \ = \ \tr_{A(-)}(\omega), 
\end{equation}
 see Figure~\ref{linear-0012}. Choose an interval on the circle, replace it by the sum of $i,i$-colored half-intervals, $1\le i\le k$, and evaluate via $\alphai$. We denote the circular series associated to a rational series $\alphai$ in this way by $\alphai^{\tr}$, so that 
\begin{equation}\label{eq_tr_alpha}
    \alphai^{\tr}(\omega) \ = \ \alpha_{\circ}(\omega) \ = \ \tr_{A(+)}(\omega). 
\end{equation}

Necessarily, $\alphai^{\tr}$ is a rational circular series in the same set $\Sigma$ of variables as $\alphai$. The coefficient of of the empty word $\emptyset$ in this series equals $\dim A(+)$. 

\vspace{0.1in} 

\input{linear-0004}

\input{linear-0012}
 
To a rational series $\alphai$ we assign  a symmetric monoidal category $\mcCp_{\alphai}$ whose objects are finite sign sequences $\varepsilon$. Morphisms from $\varepsilon$ to $\varepsilon'$ are $\kk$-linear combinations of diagrams of decorated half-intervals and outer arcs, see Figure~\ref{linear-0005}, left diagram.  
An arc is \emph{outer} if it has both endpoints on the boundary of the diagram ({\it i.e.}, among the signed points of $\varepsilon$ and $\varepsilon'$). A half-interval has one floating (inner) endpoint and one endpoint on the boundary of the diagram (outer endpoint). 
Each floating endpoint is either decorated by $i\in \{1,\dots, k\}$ or undecorated. Some evaluation rules for floating cobordisms are given in Figures~\ref{linear-0002},~\ref{linear-0004},~\ref{linear-0012}.  

\vspace{0.1in} 

\input{linear-0005}

To summarize these rules, we observe that in the category $\mcCp_{\alphai}$ 
arcs and half-intervals can be decorated by words $\omega\in\Sigma^{\ast}$. 
Half-intervals can be decorated by both a label $i$ at the floating (inner) endpoint and words $\omega$. Floating intervals and circles (these appear upon composition of morphisms) are evaluated using the following rules (see rules in Figure~\ref{linear-0003} on the right, Figures~\ref{linear-0004} and~\ref{linear-0012}):
\begin{itemize}
    \item A floating interval with unlabelled endpoints and decorated by word $\omega$ evaluates to $\alphai(\omega)$.
    Alternatively, an interval with one or two unlabelled endpoints is evaluated by first converting unlabelled endpoints into a linear combination of labelled endpoints, see equation \eqref{eq_emptyset}, or, more generally via   Figure~\ref{linear-0003} by writing $\langle \omega |$ as a linear combination of $v_1,\dots, v_k$ and $|\omega\rangle$ as a linear combination of $v^1,\dots, v^k$.  
    \item An interval with labelled endpoints $i,j$ and  decorated by a word $\omega$ evaluates to $v^i (v_j \omega)$, see Figure~\ref{linear-0004}. 
    \item A circle decorated by $\omega$ evaluates to the trace of $\omega$ on $A(+)$ or $A(-)$, see Figure~\ref{linear-0012}. 
\end{itemize}

With these evaluation rules at hand, we apply the universal construction to build the category $\mcCp_{\alphai}$. The evaluation rule for decorated circles makes decomposition of the identity in Figure~\ref{linear-0004} hold. Consequently, any outer arc reduces to a linear combination of half-intervals, with endpoint decorated by $i\in \{1,\dots, k\}$. Half-intervals with dots are further reduced to linear combinations of endpoint-decorated dotless half-intervals. 

\vspace{0.1in} 

Composing an undecorated outer arc with a possibly decorated half-interval results in a half-interval with the same decoration. Composing two half-intervals results in a floating interval, which is then evaluated via $\alphai$. 
An undecorated outer arc can be written as a linear combination of pairs of half-intervals via Figure~\ref{linear-0004} relations.  In particular, an undecorated circle evaluates to $k=\dim A(+)$, see Figure~\ref{linear-0006} left. 

\input{linear-0006}

\vspace{0.1in} 

The result is a symmetric monoidal $\kk$-linear category $\mcC_{\alphai}'$. 
It has easily describable hom spaces. A basis of $\Hom(\varepsilon,\varepsilon')$ is given by drawing the unique diagram of half-intervals ending at all signs of $\varepsilon$ and $\varepsilon'$ and adding  all possible labels $i\in\{1,\dots, k\}$ to each inner endpoint of the diagram, see an example in Figure~\ref{linear-0006} on the right. In particular, 
\begin{equation}
    \dim_{\kk} ( \Hom(\varepsilon,\varepsilon')) \ = \ k^{|\varepsilon|+|\varepsilon'|},
\end{equation}
where $|\varepsilon|$ is the length of the sequence $\varepsilon$. 
This category is a one-dimensional TQFT with defects, in the sense that the state space of the concatenation of sequences is the tensor product of state spaces for individual sequences: 
\begin{equation}
    A(\varepsilon \varepsilon') \ \cong \  A(\varepsilon) \otimes A(\varepsilon'), \ \ \ A(\varepsilon)\cong A(\epsilon_1)\otimes A(\epsilon_2)\otimes \dots \otimes A(\epsilon_n), 
\end{equation}
where $\varepsilon=\epsilon_1\epsilon_2\cdots \epsilon_n$, $\epsilon_i\in \{+,-\}.$ 

We emphasize that the category $\mcC'_{\alphai}$ depends only on the interval evaluation $\alphai$. Evaluation of decorated circles is computed as the trace of the action of $\Sigma^{\ast}$ on the state space $A(+)$ associated to $\alphai$. 

\vspace{0.1in} 

Denote by $\Kar(\mcC_{\alphai}')$ the Karoubi envelope of $\mcC_{\alphai}'$ given by allowing finite direct sums of objects of $\mcC_{\alphai}'$ and then passing to the idempotent closure. Category 
$\Kar(\mcC_{\alphai}')$ is an additive symmetric monoidal  $\kk$-linear category.

Consider idempotents $e_i \in \End_{\mcC_{\alphai}'}(+)$ given by a pair of $i$-labelled half-intervals, $i=1,\dots, k$, see Figure~\ref{linear-0016}  on the left. These are mutually-orthogonal idempotents giving a decomposition of the identity $\id_+$ endomorphism 
\begin{equation}
    \id_+ = e_1+e_2+\ldots + e_k, \hspace{1cm}  e_i e_j =\delta_{i,j} e_i, 
\end{equation}
see Figure~\ref{linear-0016} on the right. There is a similar decomposition of the identity for the dual object $-$ via idempotents $e_1',\dots, e_k'$, see Figure~\ref{linear-0016}. 

\input{linear-0016}

Note that the dual simple object $(-,e_i')$, see Figure~\ref{linear-0016} in the middle, is isomorphic to $(+,e_i)$, via the pair of morphisms show in that figure on the right. Recall that idempotent endomorphisms $e,e'$ are \emph{equivalent} (and corresponding objects of the Karoubi envelope are isomorphic) if there exist two-way composable morphisms $\beta_1,\beta_2$ such that $e=\beta_2\beta_1$ and $e'=\beta_1\beta_2$. For idempotents $e_i,e_i'$ these two morphisms are written next to the arrows between these idempotents in Figure~\ref{linear-0016} on the right. 

Furthermore, objects $(+,e_i)$ and $(+,e_j)$ are isomorphic, for $i\not=j,$ $1\le i,j\le k$, via the morphisms shown in Figure~\ref{linear-0017}. 

\vspace{0.1in} 

\input{linear-0017}

\vspace{0.1in} 

The endomorphism rings of objects $(+,e_i)$ and $(-,e_i')$ are the ground field $\kk$, and these $2k$ objects, over $i=1,\dots, k$, are pairwise isomorphic. The generating object $+$ is the sum of $k$ of them. 
This quickly leads to the following result that 
 $\Kar(\mcC_{\alphai}')$ 
 generated by $(+,e_i)$, $(-,e_i') $ over $i=1,\dots, k$  is equivalent to the tensor category $\kkvect$ of finite-dimensional $\kk$-vector spaces. 

\begin{prop} The Karoubi envelope of $\mcC_{\alphai}'$ is equivalent, as an additive symmetric monoidal category, to the 
category of finite-dimensional $\kk$-vector spaces: 
\begin{equation}
    \Kar(\mcC_{\alphai}') \ \cong \ \kkvect . 
\end{equation}
\end{prop}
\begin{proof} The equivalence is given by the functor that takes each $(+,e_i)$ to a one-dimensional vector space $V_i$, each $(-,e_i')$ to its dual $V_i^{\ast}$ and takes $+$ to the $k$-dimensional space $V=V_1\oplus \ldots\oplus V_k$. 
\end{proof} 

The proposition tells us that $\Kar(\mcC_{\alphai}')$ has a very simple structure. The complexity of noncommutative rational power series $\alphai$ is hidden in the action of $\Sigma^{\ast}$ on $A(+)\cong V$. More generally, given a TQFT taking values in  $\kk$-vector spaces, the target category of that TQFT is $\kkvect$ or some variation of it, so a version of the above proposition holds as well. 

\vspace{0.1in} 

The interesting question here is to \emph{explicitly} compute the circular evaluation $\alphai^{\tr}$ given a rational interval evaluation $\alphai$ or, equivalently, rational noncommutative power series $\alphai$. We obtain $\alphai^{\tr}$ by considering $A(+)$ and the action of $\kk\Sigma^{\ast}$ on it, so that the coefficient at $\omega$ of the noncommutative circular series of $\alphai^{\tr}$ is the trace of $\omega$ on $A(+)$, see \eqref{eq_trace_omega}, and the generating function of the circular evaluation is 
\begin{equation}\label{eq_circ_eval_gen}
    Z_{\alphai}^{\tr} \ := \ \sum_{\omega \in \Sigma^{\ast}} \  \tr_{A(+)}(\omega) \, \omega .
\end{equation}
For more than one variable, $Z_{\alpha_\I}$ is a noncommutative power series in elements of $\Sigma$. 
Circular evaluation $\alphai^{\tr}$ gives a canonical extension of $\alphai$ to a TQFT with defects. This extension is unique, in appropriate sense. It is straightforward to write down in the $1$-variable case, see Section~\ref{subsection:one-var-compare-two}.

\vspace{0.1in} 

Given a $\kk[\Sigma^{\ast}]$-module $M$, finite-dimensional over $\kk$, its characteristic function $\chi(M)$, defined in~\cite{RRV}, is given by 
\begin{equation}
    \chi(M) \ := \ \sum_{\omega\in \Sigma^{\ast}} \tr_M(\omega) \omega. 
\end{equation}
Expression \eqref{eq_circ_eval_gen} is the characteristic function of $\kk\Sigma^{\ast}$-module $A(+)$.


\subsection{One-variable case and comparison to two-dimensional theory}
\label{subsection:one-var-compare-two}

Consider the above construction in the case of a single variable, $\Sigma=\{a\}$. Then there is only one type of a dot, necessarily labelled $a$, and $n$ dots on an interval can be denoted by a single dot labelled $n$. The interval evaluation in encoded by a one-variable power series 
\begin{equation} \label{eq_one_var} 
    Z_{\I}(T) \ := \ \sum_{n\ge 0} \alpha_{\I,n}T^n .
\end{equation}
A well-known theorem, see~\cite[Proposition 2.1]{Kh3} and \cite[Theorem 2.3]{Kh2}, says that $Z_{\I}(T)$ is a rational series ($A(+)$ is finite-dimensional) if and only if it is a rational function, 
\begin{equation} \label{eq_one_var_2} 
    Z_{\I}(T) \ = \ \frac{P(T)}{Q(T)},
\end{equation}
for some polynomials $P(T),Q(T)$, with $Q(0)\not= 0$. 

\vspace{0.1in} 

We would like to explicitly compute $\alphai^{\tr}$ in this case, given the generating function above (equivalently, given interval evaluation $\alphai$). 

The state space $A(+)$ can be identified with the state space $A(1)$ in~\cite[Section 2]{Kh2} of a circle in the 2D topological theory~\cite{Kh2} associated to the same generating function $Z_{\I}(T)$. In that 2D topological theory, closed connected oriented surface of genus $n$ evaluates to $\alpha_{\I,n}\in \kk$, while in our 1D defect topological theory an interval with $n$ dots evaluates to $\alpha_{\I,n}$, see the correspondence in Figure~\ref{linear-0015} right. 

The reason is that there is a functor from the category of (oriented) dotted one-cobordisms to the category of (oriented) two-cobordisms. This functor sends $+$ (and $-$) to an oriented circle. It sends a half-interval to a disk, a  dotless arc to an annulus, and an outer arc with a single dot to a two-holed torus, see Figures~\ref{linear-0013} and \ref{linear-0014}. 

\vspace{0.1in} 

\input{linear-0013}

\input{linear-0014}

There is a bijection between homeomorphism classes of decorated connected dotted 1-manifolds with boundary $+$ and connected oriented surfaces with boundary $\SS^1$, see Figure~\ref{linear-0015} left, where $n$-dotted half-interval corresponds to a genus $n$ surface with one boundary circle. 

\input{linear-0015}

\vspace{0.1in} 

This leads to a natural isomorphism of state spaces 
\begin{equation}\label{eq_state_iso} 
A(+)\ \cong \ A(\SS^1)
\end{equation} 
of the 1D topological theory with defects with a rational generating function in \eqref{eq_one_var} for interval evaluation and the 2D topological theory with the same generating function. Notice also a canonical isomorphism $A(+)\cong A(-)$ sending upward-oriented half-interval with $n$ dots to the downward-oriented half-interval with $n$ dots, $n\ge 0$ (this isomorphism exists when $|\Sigma|=1$ and otherwise requires $\alphai$ to be invariant under word reversal). 

\vspace{0.1in}

\input{linear-0018}

One can look to compare the state spaces for these two theories (in two different dimensions) beyond a single point and a circle. In the above thickening construction, a circle with $n$ dots corresponds to a genus $n+1$ closed surface, see Figure~\ref{linear-0018}, so for the best match we pick the circular series in the 1D theory to be 
\begin{equation}\label{eq_tZ}
   \ Z_{\alphac}(T) \ = \ T Z_{\alphai}(T),  \ \ \  \alpha_{\circ,n+1} = \alpha_{\I,n},  \  n\ge 0 . 
\end{equation}
Then there is a natural $\kk$-linear map 
\begin{equation}
    A(+-) \ \stackrel{\psi_{+-}}{\lra} \ A(2) 
\end{equation}
from the state space of $+-$ in the 1D theory to that of two circles in the 2D theory, with the generating functions \eqref{eq_one_var}, \eqref{eq_tZ} for the 1D theory and \eqref{eq_one_var} for the 2D theory, given by the above thickening of dotted 1D cobordisms to 2D cobordisms. This map respects evaluations and is, in fact, an isomorphism, so that $A(+-)\cong A(2)$ in the two theories. 

More generally, for any sign sequence $\varepsilon$ there is natural map 
\begin{equation}
    A(\varepsilon) \ \stackrel{\psi_{\varepsilon}}{\lra} \ A(|\varepsilon|) 
\end{equation}
extending to a monoidal functor $\psi$ between corresponding categories for 1D and 2D evaluations. Map $\psi_{\varepsilon}$ is not surjective, for instance, for $|\varepsilon|=3$ and the constant evaluation function $Z_{\I}(T)=\beta\in \kk$, $\beta\not= 0$. Element of $A(3)$ which is the  3-holed sphere is not in the image of $\psi_{\varepsilon}$, for any length 3 sign sequence $\varepsilon$. Also, $\psi_{++}$ is not, in general, surjective, with the annulus element of $A(2)$ not in its image.   

\vspace{0.1in} 

We now come back to the problem of computing $\alphai^{\tr}$ given $\alphai$, for one-element $\Sigma=\{a\}$. Let us first examine 
two special cases. 
\begin{itemize}
    \item The generating function  $Z_{\I}(T)=\sum_{i=0}^k a_i T^i$ is a polynomial of degree $k$. Then $\dim A(+)=k+1$ and the operator of multiplication by $a$ is nilpotent. Consequently, circular evaluation $\alphac$ is the constant function, taking value $0$ on any nonzero power of $a$, with the generating function $Z_{\circ}(T)=k+1$. 
    \item The generating function is a reduced fraction of the 
    form 
    \begin{equation}\label{eq_red_frac_1}
        Z_{\I}(T) = \frac{f(T)}{(\lambda-T)^k}, \ \  k\ge 1,\ \  \lambda\not= 0, \ \ \deg(f(T))<k. 
    \end{equation}
    Then $A(+)$ is a cyclic $\kk[a]$-module isomorphic to $\kk[a]/((\lambda a -1)^k)$,  where we quotient by the reciprocal polynomial of $(\lambda-a)^k$, see~\cite{Kh3}. 
    Note that trace of $a^m$ on this quotient space does not depend on $f(T)$ above, subject to the conditions in \eqref{eq_red_frac_1}. 
     Substituting $u=\lambda a -1$, so that $a=\frac{1}{\lambda}(u+1)$, the trace of $a^m$ on $\kk[u]/(u^k)$ is given by  
    \begin{equation}  \tr(a^m) \ = \ \lambda^{-m} k .
    \end{equation}
 \end{itemize}
 
    A rational function, over an algebraically closed field $\kk$, has a unique partial fraction decomposition
    \begin{equation}\label{eq_dec_part}
       Z_\I(T)= \frac{P(T)}{Q(T)} \ = \ \sum_{i=1}^r \frac{f_i(T)}{(\lambda_i-T)^{k_i}} + f_0(T), \ \ \  \lambda_i\not=0,\  \deg(f_i(T))<k_i, \ i=1,\dots, r. 
    \end{equation}
    Then the trace circular series associated with this generating function is 
    \begin{eqnarray} \label{eq_Z_trace}
        Z^{\tr}_{\alphai}(T)   & = &  \deg(f_0)+1 + \sum_{m\ge 0} \left(\sum_{i=1}^r k_i \lambda_i^{-m}  T^m \right)  \\
        \nonumber
        & = & \deg(f_0)+1 +  \sum_{i=1}^r k_i \sum_{m\ge 0}(\lambda_i^{-1} T)^m  \\
        \nonumber
        & = & \deg(f_0)+1 +  \sum_{i=1}^r \frac{k_i}{1-\lambda_i^{-1} T}
    \end{eqnarray}
If $Z_{\I}(T)$ in \eqref{eq_dec_part} is a proper fraction, that is, $f_0(T)=0$, we set $\deg(f_0)+1=0$ in \eqref{eq_Z_trace}.
Note that the characteristic polynomial for the trace series is a divisor of the characteristic polynomial for the original series.  
    
    If eigenvalues of $a$ on $A(+)$ are $\mu_1,\dots, \mu_n$, listed with multiplicities, then 
    \begin{equation}\label{eq_Z_mu} 
        Z^{\tr}_{\alphai}(T) \ = \ \sum_{i=1}^n \frac{1}{1-\mu_i T},
    \end{equation}
    which is~\cite[Example 2.6]{RRV}. 
    
    It is an interesting problem to explicitly write down noncommutative trace series $Z_{\alphai}^{\tr}(\Sigma)$ associated with an arbitrary rational noncommutative series $Z_{\alphai}(\Sigma)$ when the number of variables $|\Sigma|$ is greater than one.


\subsection{A topological theory when a circular series is added}
\label{subsec_top_theory}

To build a more general monoidal category, we additionally pick a circular rational series $\alpha_{\circ}$, see~\cite{Kh3,IK-automata}. Here $\alpha=(\alpha_\I,\alpha_{\circ})$ is a pair: a rational noncommutative series $\alpha_\I$ and a circular rational noncommutative series $\alpha_{\circ}$. We build category $\mcC_{\alpha}$ from it as in Section~\ref{subsection_oned_defectTQFT} by evaluating floating decorated intervals and circles via $\alpha_\I$ and $\alphac$ correspondingly and applying the universal construction to derive further relations on linear combinations of cobordisms with outer boundary, see also~\cite{Kh3, Kh2,KS3}. As before, objects of $\mcC_{\alpha}$ are finite sign sequences $\varepsilon$. 

\vspace{0.1in} 

In the category $\mcC_{\alpha}$, state spaces $A(+),A(-)$ depend only on $\alpha_\I$ and they are spanned by elements $|\omega\rangle$, $\omega\in \Sigma^{\ast}$, respectively $\langle \omega |$, $\omega\in \Sigma^{\ast} $. State space $A(+-)$ is spanned by diagrams of two types:
\begin{enumerate}
\item\label{item:state-space+-span-01}
pairs of decorated half-intervals with opposite orientations, 
\item\label{item:state-space+-span-02} decorated outer arcs, 
\end{enumerate}
see Figure~\ref{linear-0007}.

\input{linear-0007}

$A(+-)$ is naturally a unital associative finite-dimensional algebra, with the unit element given by an outer arc with the trivial decoration and with the  multiplication shown in Figure~\ref{linear-0008}. Algebra $A(+-)$ acts on $A(+)$ on the left, see Figure~\ref{linear-0008}. 

\input{linear-0008}

Denote by $\mmcI$ 
the subspace of $A(+-)$ spanned by diagrams of  type~\eqref{item:state-space+-span-01}, see Figure~\ref{linear-0007} second from left picture. This subspace is a two-sided ideal of $A(+-)$ and a unital $\kk$-algebra with the unit element  $1'=\sum_{i=1}^k v^i\otimes v_i$ shown in Figure~\ref{linear-0004} on the right hand side of the equality. Note that the equality $1'=1$ fails unless $\alpha_{\circ}=\alphai^{\tr}$. 
In general, the right hand side diagram is the unit element of $A(+-)$ and the left hand side is the idempotent $1'$. 

There is a natural algebra isomorphism  
\begin{equation}
    \mmcI \cong A(+)\otimes A(-) \cong \End(A(+))   
\end{equation}
coming from the faithful action of $\mmcI$ on $A(+)$, given by restricting the action from that of $A(+-)$.  

The kernel $\mmcK$ of the action of $A(+-)$ on $A(+)$ is a  two-sided ideal of $A(+-)$, complementary to $\mmcI$, giving 
a direct product decomposition 
\begin{equation}\label{eq_product}
    A(+-) \ \cong \ \mmcI \times \mmcK. 
\end{equation}
In this decomposition both terms on the right are unital $\kk$-algebras, with the unit element of $\mmcK$ given by the image of $1\in A(+-)$ under the projection, that is, by 
\begin{equation}\label{eq_oneK} 
1_{\mmcK} \ := \ 1-1' = 1-\sum_{i=1}^k v^i\otimes v_i.
\end{equation}
In particular, $1_{\mmcK}$ generates $\mmcK$ as an $A(+-)$-bimodule.

Denote by $U$ the subspace of $A(+-)$ spanned by diagrams of type~\eqref{item:state-space+-span-02}, that is, by decorated arcs connecting $+$ and $-$ boundary points, see the picture on the right in Figure~\ref{linear-0007}. Recall that we denote by $\circleft(\omega)$ the arc decorated by $\omega$, see Figure~\ref{linear-0007} on the right.
$U$ is a unital subalgebra of $A(+-)$ and there are algebra inclusions
\begin{equation}
      U \subset A(+-) \supset A(+)\otimes A(-)\cong \mcI .
\end{equation}
Subalgebra $U$ surjects onto $\mmcK$ upon projection to the second term in the direct product \eqref{eq_product}, and there is a short exact sequence
\begin{equation} \label{eq_exact_U}
    0 \lra U \cap \mmcI \lra U \lra \mmcK \lra 0 . 
\end{equation}
The first term 
\begin{equation} \label{eq_U_prime} 
    U' \ := \ U \cap \mmcI 
\end{equation} 
is a two-sided ideal of $U$. Elements in $U'$ are linear combinations of decorated arcs (elements of $U$) that decompose in $A(+-)$ into linear combinations of pairs of half-intervals. These decompositions are unique as elements of $A(+)\otimes A(-)\cong \mmcI\subset A(+-)$.

This data carries a triple of discrete invariants: 
\begin{equation}\label{eq_triple}
    (\dim A(+), \dim (U'), \dim \mmcK), 
\end{equation}
which are three non-negative integers. 
Note that $\dim A(-)=\dim A(+)$, $\dim \mmcI=(\dim A(-))^2$ and $\dim (U')\le (\dim A(+))^2$. 

\begin{remark} \label{rem_repeat_1}
The natural inclusion $A(+)\otimes A(-)\subset A(+-)$ is an isomorphism if and only if $\mmcK=0$. 
This is exactly the case when there is a decomposition of the identity, 
that is, when the undecorated arc $1\in A(+-)$ lies in $\mmcI \cong A(+)\otimes A(-)$, that is, when $1$ is \emph{decomposable} (also when $1=1'$, see earlier).  
In this case $A(+-)\cong A(+)\otimes A(-)$ and, more generally, $A(\varepsilon\varepsilon')\cong A(\varepsilon)\otimes A(\varepsilon')$ for any sign sequences $\varepsilon,\varepsilon'$.

Equivalently, the category $\mcC_{\alpha}$ gives a TQFT rather than just a topological theory (with defects) if and only if $\mmcK=0$.   This is possible for a unique rational circular series $\alphac=\alphai^{\tr}$ associated with $\alpha_\I$ and with the identity decomposition determined by $\alpha_\I$, see formula \eqref{eq_tr_alpha}.  Circular evaluation $\alpha_\circ$ is then the trace of action of words on $A(+)$, see formula \eqref{eq_circ_eval_gen}, with $\alphac(\omega)=\alphai^{\tr}(\omega)=\tr_{A(+)}(\omega)$ for $\omega\in \Sigma^{\ast}$.  
The resulting theory $\alpha=(\alphai,\alphac)$ is a TQFT.

The Boolean version of the identity decomposition (which requires $A(+)$ to be a distributive semilattice) is considered in \cite{IK-automata}. 
\end{remark}

\begin{example} \label{ex_eval_1}
Consider the case when $\Sigma$ is empty. Then there are only two closed connected cobordisms, undecorated interval and circle. Suppose they evaluate to $1$ and $\lambda\not=1$, respectively, see Figure~\ref{linear-0009}, top middle. Then $A(+)$ and $A(-)$ are one-dimensional, with basis vectors $v$ and $v^{\ast}$, respectively, see Figure~\ref{linear-0009}, top right. 
The space $A(+-)$ is two-dimensional, with the basis shown in Figure~\ref{linear-0009}, top left. The subalgebra $\mmcK$ is one-dimensional,
$\mmcK=\kk w$, with a basis element $w$ shown in Figures~\ref{linear-0009}, bottom left, and~\ref{linear-0010}. The algebra $U$ is one-dimensional, with the unit element (an arc) as the basis element, see Figure~\ref{linear-0010}. 

In this example $U\cap \mmcI=U\cap \mmcK=0$, and projection of $U$ onto $\mmcK$ along $\mmcI$ is an algebra isomorphism, See Figure~\ref{linear-0010}. The triple of parameters is $(1,0,1)$. There is algebra decomposition $A(+-)\cong \mmcI\times \mmcK = \kk(1-w)\times \kk w.$

\input{linear-0009}

\input{linear-0010}

\end{example} 


When $\Sigma=\{a\}$ is a one-element set, decoration of an interval or a circle is determined by the number $n\ge 0$ of dots on it. If an $n$-dotted interval evaluates to $\alpha_{\I,n}$ and $n$-dotted circle evaluates to $\alpha_{\circ,n}$, the evaluation is encoded by a pair of one-variable power series in a variable $T$: 
\begin{equation}\label{eq_encode} 
    Z_{\I}(T) \ := \ \sum_{n\ge 0} \alpha_{\I,n}T^n, \ \ \ 
    Z_{\circ}(T) \ := \ \sum_{n\ge 0} \alpha_{\circ,n}T^n,
\end{equation}
\input{linear-0011}
see Figure~\ref{linear-0011}.
A simple modification of a  result from~\cite{Kh3} shows the following. 
\begin{prop} A one-variable evaluation $\alpha=(\alpha_\I,\alpha_{\circ})$ is rational (the hom spaces in the category $\mcC_{\alpha}$ are finite-dimensional) if and only if both $Z_\I(T)$ and $Z_{\circ}(T)$ are rational functions in $T$, \textit{i.e.}, 
\begin{equation}\label{eq_generating}
    Z_\I(T) = \frac{P_{\I}(T)}{Q_{\I}(T)}, \hspace{1.0cm}  Z_{\circ}(T) = \frac{P_{\circ}(T)}{Q_{\circ}(T)}. 
\end{equation}
\end{prop} 
It is also easy to see that $\alpha$ is a rational evaluation if and only if the state spaces $A(+),A(+-)$ are finite-dimensional. 

\vspace{0.1in} 

\begin{example} \label{ex_2}
Let us construct an example with parameters $(1,1,1)$ as in \eqref{eq_triple}. Since $\dim(U)=2$, take $\Sigma=\{a\}$. Since $\dim A(+)=1$, the dot acts by some scalar $s$ on the endpoint, see Figure~\ref{linear-ex-0001}. Pick evaluation of undecorated interval and circle to be $1$ and $\lambda\not=1 $, respectively, and introduce parameter $t$ for the skein relation in Figure~\ref{linear-ex-0001} top right. 

\input{linear-ex-0001}

Attaching half-interval at the top right endpoint in each diagram of the relation implies $s=t+1$. Closing up the skein relation by an arc with $n$ dots gives an inductive formula for a circle with $n$ dots. Generating functions \eqref{eq_generating} are
\begin{equation}\label{eq_generating_ex2}
    Z_\I(T) = \frac{1}{1-(t+1)T},  \hspace{1cm}  Z_{\circ}(T) = \frac{\lambda-1}{1-tT} + \frac{1}{1-(t+1)T}. 
\end{equation}
Inductive skein relation to reduce the number of dots is shown in Figure~\ref{linear-ex-0002} on the left. Figure~\ref{linear-ex-0002} on the right shows the unit element and the basis vector of algebra $\mmcK\cong \kk$, equal to $1-1'$. 

\input{linear-ex-0002}

We get a decomposition $A(+-)\cong \mmcI\times \mmcK\cong \kk\times \kk$ of $A(+-)$ into the product of two copies of the ground field. 

\end{example} 

\begin{example}\label{ex_3}
Let us give an example with  $\dim A(+)=2$, $\dim(U')=1$, 
$\dim(\mmcK)=1$  and $U'=\mc{I}\cap U$, a  nilpotent ideal in $U$.
Let $\Sigma=\{a\}$ and make $a$ nilpotent, $a^2=0\in U$, see Figure~\ref{linear-ex-0003}. Then $\{\langle \emptyset|, \langle a |\}$ is a basis of $A(+)$ and for the bilinear pairing $A(+)\times A(-)\lra \kk$ we can choose the one in Figure~\ref{linear-ex-0003}, with a parameter $\mu\in \kk$. Additionally, choose a relation reducing dot on an arc to a pair of dotted half-intervals. Evaluation $\alpha$ is then shown in Figure~\ref{linear-ex-0003} second and bottom rows. 

\input{linear-ex-0003}

\input{linear-ex-0007} 

Space $A(+-)$ is spanned by the six vectors, with the pairing given in Figure~\ref{linear-ex-0004} (rows 4 and 6 are equal and columns 4 and 6 are equal). Dropping the last row and column gives us a basis of $A(+-)$ of cardinality $5$. See Figure~\ref{linear-ex-0007}. 

\input{linear-ex-0004}
 
The structure of the short exact sequence \eqref{eq_exact_U} for this example is shown in Figure~\ref{linear-ex-0005} left. 
The generating functions are 
\begin{equation}\label{eq_generating_ex3}
    Z_\I(T) = \mu+T,  \hspace{1cm}  Z_{\circ}(T) = \lambda,
\end{equation}
see also Figure~\ref{linear-ex-0003}. 

\input{linear-ex-0005}

\input{linear-ex-0006} 

\end{example}


\subsection{Karoubi envelope decomposition for arbitrary rational \texorpdfstring{$\alpha$}{alpha}}\label{subsec_karoubi_env} 

We now go back to the case of arbitrary $\Sigma$ of finite cardinality and a rational pair $\alpha=(\alphai,\alphac)$. Recall the direct product decomposition of $A(+-)$ into $\mmcK$ and  the matrix ring $\mmcI$, and the formula for the unit element of $\mmcK$:  
\begin{equation} \label{eq_recall} 
A(+-) \cong \mmcI \times \mmcK, \ \ \mmcI\cong \End_{\kk}(A(+)), \ \ 1_{\mmcK} = 1 - \sum_{i=1}^k v^i\otimes v_i. 
\end{equation}  
Elements of $\mmcK$ placed on arcs act trivially on $A(+)$ and $A(-)$, see Figure~\ref{linear-0019}. In the earlier examples, elements of $\mmcK$ are shown in Figures~\ref{linear-0009},~\ref{linear-ex-0002},~\ref{linear-ex-0006}, where in each case $\dim \mmcK=1$. 

\vspace{0.1in}

\input{linear-0019}

\vspace{0.1in} 

Denote by 
\begin{equation}\label{eq_p} 
    p \ : \ U \subset A(+-) \stackrel{p_{\mmcK}}{\lra} \mmcK
\end{equation}
the composition of the inclusion of algebra $U$ into $A(+-)$ and projection $p_{\mmcK}$ onto $\mmcK$ along $\mmcI$. The latter map can be written as either left or right multiplication by $1_{\mmcK}$, 
\begin{equation}\label{eq_p_U} 
    p_{\mmcK}(y) = 1_{\mmcK} \, y = y 1_{\mmcK},  
\end{equation}
and, hiding the inclusion, we can write $p(x) = 1_{\mmcK}\, x = x 1_{\mmcK}$ for $x\in U$. 

Furthermore, the surjection $\kk \Sigma^{\ast}\lra U$ can be composed with $p$ above. Denote the resulting algebra homomorphism  by 
\begin{equation}\label{eq_p_sigma} 
p_{\ast}:\kk\Sigma^{\ast}\lra \mmcK,  \ \ p_{\ast}(\omega) = 1_{\mmcK}\omega = \omega 1_{\mmcK}, \ \ \omega\in \Sigma^{\ast}, 
\end{equation} 
where, when multiplying by $1_{\mmcK}$,  we view $\omega$ as an element of $U$ or $A(+-)$. 

Introduce a a trace map $\tr$ on $\mmcK$ by the circular closure and evaluation, 
see Figure~\ref{linear-0020}, 
\begin{equation}\label{eq_tr_on_K} 
    \tr \ : \ \mmcK \lra \kk. 
\end{equation} 
Note that, in general, both $\alphai$ and $\alphac$ are used for the evaluation, since elements of $\mmcK$ are linear combinations of elements of $U$ and $A(+)\otimes A(-)$. The circular closure on elements on $U$, respectively $A(+)\otimes A(-)$, is computed via $\alphac$, respectively $\alphai$. 

\vspace{0.1in} 

\input{linear-0020}

The following relation between traces holds: 
\begin{equation}\label{eq_tr_pstar}
     \tr (p_{\ast}(\omega)) \ = \ \alphac(\omega) - \alphai^{\tr}(\omega),  \ \  \omega\in \Sigma^{\ast}, 
\end{equation}
where 
\begin{equation}
    \alphai^{\tr}(\omega) \ := \   \tr_{A(+)}(\omega). 
\end{equation}
Recall that evaluation $\alphai^{\tr}$ depends only on the interval evaluation $\alphai$ and is given by the trace of words on $A(+)$. 

Note that $\tr(\omega)=\alphac(\omega)$, since we view $\omega$ as an element of either $\kk\Sigma^{\ast}$ or $U$, via the projection onto the latter. In formula \eqref{eq_tr_pstar} projection $p_{\ast}$ appears, which introduces an additional term. 

\begin{remark}
If comparing to formula \eqref{eq_tr_alpha}, recall that there circular series $\alphac$ was picked to depend on $\alphai$ and give a 1D TQFT, while here we are considering arbitrary rational $\alphai,\alphac$. 
\end{remark}

Extending linearly from $\omega$ to elements of $\kk\Sigma^{\ast}$, we get 
\begin{equation}\label{eq_trace_K}
     \tr (p_{\ast}(z)) \ = \ \alphac(z) - \alphai^{\tr}(z),  \hspace{0.75cm}   z \in \kk \Sigma^{\ast}, 
\end{equation}

\begin{prop} Trace $\tr$ in formula \eqref{eq_tr_on_K} turns $\mmcK$ into a symmetric Frobenius algebra with the unit element $1_{\mmcK}$.  
\end{prop} 
A Frobenius algebra is called \emph{symmetric} if $\tr(xy)=\tr(yx)$ for any elements $x,y$. 
\begin{proof} 
The trace \eqref{eq_tr_on_K} on $\mmcK$ is symmetric, since the circular closure is symmetric. The pairing $A(+-)\otimes A(+-) \lra \kk$ is non-degenerate, and elements of $\mmcK$ are orthogonal to those of  $\mmcI = A(+)\otimes A(-)$ in this pairing. Consequently and in view of the direct product (hence direct sum) decomposition \eqref{eq_recall}, any nonzero element of $\mmcK$ can be paired to some element of $\mmcK$ to get a nontrivial evaluation, implying that the trace is non-degenerate. 
\end{proof} 

\begin{remark}\label{rem_symm}
Vice versa, any symmetric Frobenius algebra $(B,\tr_B)$ can be obtained from some rational evaluation $\alpha=(\alphai,\alphac)$ in this way. For that, choose a set $\Sigma$ of generators of algebra $B$ and form circular series $\alpha_{\circ}(\omega)=\tr_B(\omega)$, viewing $\omega\in\Sigma^{\ast}$ as an element of $B$ via the monoid homomorphism $\Sigma^{\ast}\lra B$, where $B$ is naturally a monoid under multiplication. Set the interval evaluation $\alphai$ identically to zero, $\alphai(\omega)=0$ for any word $\omega$. Then $A(+)=0=A(-)$, 
the trace $\alphai^{\tr}=0$, and $\mmcK=A(+-)\cong B$ with the trace $\tr$. Similar examples can be produced with $A(+)\not= 0$.  
\end{remark}

Pair $(B,\tr_B)$ as in Remark~\ref{rem_symm} gives rise to monoidal category $\mcC_B$ obtained via the universal construction as follows. First consider the category $\mcC'_B$ with objects -- finite sign sequences and morphisms finite $\kk$-linear combinations of one-dimensional cobordisms with defects. Floating endpoints are not allowed this time, and defects are labelled by elements of $B$, see Figure~\ref{linear-0025} on the left. Concatenation and addition of defects corresponds to multiplication and addition in $B$, and decorated circles are evaluated via the trace on $B$, see Figure~\ref{linear-0025}. 
\input{linear-0025}

Since $B$ is not, in general, commutative, it is essential to require strands to be oriented, to make sense of the concatenation formula in Figure~\ref{linear-0025} (second diagram from left). 
The resulting category $\mcC_B'$ can be thought as a decorated version of the oriented Brauer category and is known as the \emph{Frobenius--Brauer category}~\cite{SS22,MS22,Savage-notes-21}. Next, we apply the universal construction to $\mcC'_B$ (since trace $\tr_B$ allows to evaluate any closed diagram) to get the quotient category, denoted $\mcC_B$. Equivalently, one can define $\mcC_B$ as the quotient of $\mcC_B'$ by the ideal of negligible morphisms. Trace on $B$ is omitted from our notations for these two categories, but both $\mcC'_B$ and $\mcC_B$ depend on it as well. 

\begin{remark} Algebra $A(+-)$ carries two trace maps, see Figure~\ref{linear-0024}. This pair of maps is nondegenerate on $A(+-)$, in a suitable sense. In the above construction only the first trace map is considered.

\input{linear-0024}

\end{remark} 

\begin{remark} It is interesting that starting with a 1-dimensional theory with defects one obtains a symmetric Frobenius algebra $\mmcK$, since the latter describes a 2D TQFT for thin surfaces, see \cite[page 19]{KQ-notes-20} as well as \cite{LP08,LP09,MS06,lauda2006open,KQR},  hinting at a sort of dimensional lifting. We discuss this later in the paper, in Section~\ref{section_2d}. 
\end{remark}

Algebra $A(+-)$ has an idempotent decomposition \eqref{eq_recall}, which we can write as 
\begin{equation}
    1 \ = \ 1_{\mmcK} + \sum_{i=1}^k v^i\otimes v_i, 
\end{equation}
with each $v^i\otimes v_i$, $1\le i\le k$ and $1_{\mmcK}$ together constituting $k+1$ mutually orthogonal idempotents. 
There are natural algebra isomorphisms 
\begin{equation}\label{eq_bend_alg}
     \End_{\mcC_{\alpha}}(+) \ \cong \ A(+-) \cong \End_{\mcC_{\alpha}}(-)^{\mathrm{op}}. 
\end{equation}
given by bending strands, see Figures~\ref{linear-0022} and~\ref{linear-0023}.

\input{linear-0022}

\vspace{0.1in} 
 
\input{linear-0023}

The corresponding orthogonal idempotent decompositions in $\End_{\mcC_{\alpha}}(+)$ and $\End_{\mcC_{\alpha}}(-)^{\mathrm{op}}$ are given by 
\begin{equation}\label{eq_id_p_m}
    \id_+ \ = \ 1^+_{\mmcK} + \sum_{i=1}^k v^i \otimes v_i^{\ast}, \ \ \ \ \ 
    \id_- \ = \ 1^-_{\mmcK} + \sum_{i=1}^k v_i \otimes v^i_{\ast},
\end{equation}
see Figure~\ref{linear-0021}

\vspace{0.1in} 

\input{linear-0021}

\vspace{0.1in} 

Consider the Karoubi envelope $\KarCa$ of $\mcC_{\alpha}$,  also denoted $\undDC_{\alpha}$ in~\cite{Kh3}. Denote by $e^+_0, e^+_1, \dots, e^+_k$ the idempotents in the endomorphism ring of $+$, see \eqref{eq_id_p_m}, so that $e_0^+=1^+_{\mmcK}$ and $e_i=v^i\otimes v_i^{\ast}$, and use the same notation but with $-$ instead of $+$  for the $-$ object, so that 
\begin{equation}\label{eq_idemp_pm}
    \id_+ = e^+_0+ \ldots + e^+_k , \ \ \ \ 
    \id_- = e^-_0+ \ldots + e^-_k
\end{equation}
Objects $+$  and $-$  in the category $\KarCa$ are isomorphic to the direct sum of objects  
\begin{equation}
    + \cong \oplus_{i=0}^k (+, e^+_i), \ \ \ \ - \cong \oplus_{i=0}^k (-, e^-_i). 
\end{equation}
For $0\le i \le k$ objects $(+,e^+_i)$ and $(-,e^-_i)$ are dual. Furthermore, we have

\begin{prop}
Objects $(+,e_i^+)$ and $(-,e_i^-)$ are isomorphic, for $i\ge 1$. Objects $(+,e_i^+)$ and $(+,e_j^+)$ are 
isomorphic, for $i,j\ge 1$. Each of these objects is isomorphic to the identity object $\oneb$ of $\Kar(\mcC_{\alpha})$.
\end{prop} 
\begin{proof} It is enough to set up pairs of morphisms between the corresponding objects of $\mcC_{\alpha}$ such that the compositions are the corresponding idempotents. These morphisms are shown in Figure~\ref{linear-0017} for objects $(+,e_i^+)$ and $(+,e_j^+)$. For objects $(+,e_i^+)$ and $(-,e_i^-)$ the morphisms are shown in Figure~\ref{linear-0016} on the right. Note also that $\End_{\KarCa}((+,e_i^+))=\kk$. 
Isomorphisms between each of $(+,e_i^+)$ and $(-,e_i^-)$ and $\oneb$ are given by the half-intervals (up oriented for $+$, down oriented for $-$)  with the top or bottom boundary $+$ or $-$. The identity object $\oneb$ is represented by the empty $0$-manifold  $\emptyset_0$ (by the empty sign sequence). 
\end{proof}

\begin{corollary}
  $2k$ objects $(+,e_1^+),\dots,(+, e_k^+)$ and $(-,e_1^-),\dots,(-, e_k^-)$ are pairwise isomorphic, and each is isomorphic to $\oneb$. For any of these objects $X$ the endomorphism ring $\End_{\KarCa}(X)\cong \kk$. 
\end{corollary}

Denote by $\mcC'$ the full tensor additive subcategory of $\KarCa$ generated by these $2k$ objects and $\oneb$. It is a rigid category. 

\begin{prop} The category $\mcC'$ is tensor (symmetric monoidal) equivalent to the category $\kkvect$ of finite-dimensional $\kk$-vector spaces with the standard tensor structure. It is Karoubi-closed. 
\end{prop} 

\begin{proof} This is immediate since monoidal generators of $\mcC'$ are all equivalent to $\oneb$. The unit object $\oneb\in \Ob(\KarCa)$ generates a full monoidal subcategory of $\KarCa$ which is monoidal equivalent to $\kkvect$. 
\end{proof}

Recall the complementary idempotent $e_0^+=1_{\mcK}^+$ to $e_{>0}^+:=e_1^+ + \ldots + e_k^+$  in $\id_+$, see Figure~\ref{linear-0021}. Likewise, $e_0^-=1_{\mcK}^-$ is the complementary idempotent to $e_{>0}^-:=e_1^-+\ldots + e_k^-$ in $\id_-$.  
Consider the corresponding objects of $\KarCa$: 
\begin{equation}
 X_0^+ \ := \ (+,e_0^+),  \ \ \ X_0^- \ := \ (-,e_0^-), \ \ \ 
 X_{>0}^+ \ := \ (+,e_{\ast}^+), \ \ \ X_{>0}^- \ := \ (-,e_{\ast}^-). 
\end{equation}
Then the first pair of objects is \emph{monoidal orthogonal} to the second pair in the following strong sense. For any $n>0$, $m\ge 0$ 
\begin{eqnarray}
    \Hom_{\KarCa}((X_0^+\oplus X_0^-)^{\otimes n}, (X_{>0}^+\oplus X_{>0}^-)^{\otimes m}) & = &  0, \label{eq_ort1} \\
    \Hom_{\KarCa}((X_{>0}^+\oplus X_{>0}^-)^{\otimes m},(X_0^+\oplus X_0^-)^{\otimes n}) & = &  0, \label{eq_ort2}
\end{eqnarray}
that is, the space of homs between any nonempty finite tensor product of $X_0^+$ and $X_0^-$ and a finite tensor product of $X_{>0}^+$ and $X_{>0}^-$ is zero. In particular, there is only the $0$ morphism between any nonempty finite tensor product of $X_0^+$ and $X_0^-$ and $\oneb$. 

Equations \eqref{eq_ort1}, \eqref{eq_ort2} follow from the orthogonality between elements of $\mcK$ and elements of $A(+)$, $A(-)$ shown in Figure~\ref{linear-0019} on the left and center. 

Objects $X_0^+$ and $X_0^-$ are dual. Denote by $\wmcCa$ the full monoidal additive and Karoubi-closed subcategory of $\KarCa$ generated by these two objects (and object $\oneb$). We now relate this category to the Frobenius--Brauer category $\mcC'_{\mcK}$ associated to the Frobenius algebra $(\mcK,\tr_{\mcK})$ and its negligible quotient $\mcC_{\mcK}$, see earlier. 

\vspace{0.1in}

Recall that to a symmetric Frobenius $\kk$-algebra $(B,\tr_B)$, we have assigned a category $\mcC'_B$ (the Frobenius--Brauer category) of 1-cobordisms with dots decorated by elements of $B$ subject to relations in Figure~\ref{linear-0025}. Then category $\mcC_B$ is the quotient of $\mcC'_B$ via the universal construction for the evaluation of closed $B$-decorated 1-manifold via $\tr_B$. Equivalently, $\mcC_B$ is the \emph{gligible quotient} of $\mcC'_B$, the quotient by the 2-sided ideal of negligible morphisms. 

Now specialize to the symmetric Frobenius algebra $(\mcK,\tr_{\mcK})$ associated to the evaluation $\alpha$. 
We have canonical $\kk$-algebra isomorphisms $\End(X_0^+)\cong \mcK\cong \End(X_0^-)$. 

\begin{prop} These isomorphisms extend to a monoidal and full functor 
\begin{equation}\label{eq_fun_F0}
    \mcF_0 \ : \ \mcC'_{\mcK} \lra \KarCa
\end{equation}
taking the object $+\in \Ob(\mcC'_{\mcK})$ to $X_0^+=(+,e_0^+)$ and object $-$ to $X_0^-=(-,e_0^-)$. 
\end{prop} 
\begin{proof} The functor $\mcF_0$ takes $\oneb$ to $\oneb$. It takes an arc carrying a dot labelled $x\in\mcK$ to the arc with the same label, which is now viewed as a morphism in $\KarCa$ between products of $X_0^+$ and $X_0^-$. For instance, The arc in Figure~\ref{linear-0024} on the left, for $x\in \mcK$, can be viewed as a morphism in $\mcC'_{\mcK}$ from $\oneb$ (the identity object in $\mcC'_{\mcK}$) to the object $+-$ and, alternatively, as a morphism in $\KarCa$ from $\oneb$ (the identity object in $\KarCa$) to the object $X_0^+\otimes X_0^-$. 

Earlier computations, including orthogonality relations \eqref{eq_ort1}, \eqref{eq_ort2} imply that $\mcF_0$ is well-defined and surjective on morphisms (a full functor). 
\end{proof} 

Category $\KarCa$ is the Karoubi envelope of $\mcC_{\alpha}$, the latter defined via the universal construction for the evaluation $\alpha$. Restricting evaluation $\alpha$ to closures of elements of $\mcK\subset A(+-)$ results in the trace $\tr_{\mcK}$ on $\mcK$ used in the construction of the category $\mcC'_{\mcK}$ and its gligible quotient $\mcC_{\mcK}$ (quotient via the universal construction). The two evaluations -- in categories $\mcC_{\alpha}$ and $\mcC'_{\mcK}$ -- match, for the above inclusion $\mcK\subset A(+-)$ . Consequently, we obtain the following statement. 

\begin{prop} \label{prop_F0} Functor $\mcF_0$  in \eqref{eq_fun_F0} factors through the gligible quotient $\mcC_{\mcK}$ of $\mcC'_{\mcK}$ and induces a fully-faithful and monoidal functor 
\begin{equation}
    \mcF \ : \ \mcC_{\mcK} \lra \KarCa
\end{equation}
\end{prop} 
In particular, functor $\mcF_0$ factorizes as the composition 
\begin{equation}
    \mcC'_{\mcK} \lra \mcC_{\mcK} \stackrel{\mcF}{\lra} \KarCa.
\end{equation}
where the first arrow is the gligible quotient functor. 

Functor $\mcF$ induces a monoidal functor 
\begin{equation}\label{eq_equiv_Kar}
    \Kar(\mcF) \ : \ \Kar(\mcC_{\mcK}) \ \lra \Kar(\mcC_{\alpha}). 
\end{equation}
\begin{prop}
  Functor $\Kar(\mcF)$ is an equivalence of categories. 
\end{prop}
\begin{proof}
  Recall the complementary object $X_{>0}^+$ to $X_0^+$ in $+\in \Ob(\Kar(\mcC_{\alpha})$, so that $+\cong X_0^+\oplus X_{>0}^+$. Likewise, $X_{<0}^-$ is complementary to $X_0^-$ with respect to the object $-$ of $\Kar(\mcC_{\alpha})$. Objects $X_{>0}^+$ and $X_{<0}^-$ are both isomorphic to the direct sum of $k$ copies of the object $\oneb$. Consequently,  the category $\Kar(\mcC_{\alpha})$ is equivalent to the Karoubi closure of the monoidal subcategory generated by $X_0^+$ and $X_0^-$. Taking earlier propositions into consideration completes the proof. 
\end{proof}

We now explain the meaning of these results. Starting with a rational evaluation $\alpha=(\alphai,\alphac)$ we formed the decomposition $A(+-)\cong \mcI\times \mcK$, with $\mcI\cong A(+)\otimes A(-)$ and the complementary factor $\mcK$ orthogonal, in a suitable sense, to $A(+)$ and $A(-)$. 

Algebra $\mcK$ is symmetric Frobenius, with the nondegenerate trace given by the closure operation in Figure~\ref{linear-0020}. Note also the trace formula \eqref{eq_tr_pstar} for the trace computed on the projection from $\kk\Sigma^{\ast}\subset A(+-)$ onto $\mcK$ as the difference $\alphac-\alphai^{\tr}$. 


Symmetric Frobenius algebras (Frobenius algebras $B$ with a symmetric trace, $\tr(xy)=\tr(yx), x,y\in B$) are also called just \emph{symmetric} algebras. 
To the symmetric algebra $(\mcK,\tr_{\mcK})$ we assign the Frobenius--Brauer category $\mcC'_{\mcK}$ and its gligible quotient $\mcC_{\mcK}$. Then there is a natural fully-faithful monoidal functor $\mcF:\mcC_{\mcK}\lra \Kar(\mcC_{\alpha})$, 
see Proposition~\ref{prop_F0}, inducing an equivalence \eqref{eq_equiv_Kar} of Karoubi envelopes of $\mcC_{\mcK}$ and $\mcC_{\alpha}$.  

Conceptually, passing to the Karoubi envelope $\Kar(\mcC_{\alpha})$ allows one to reduce the consideration to the Frobenius algebra $\mcK$ together with the trace $\tr_{\mcK}$ on it. Decorated half-arc morphisms go between $\oneb$ and summands of $+$ and $-$ objects which are equivalent to $\oneb$ and can be ignored in the Karoubi envelope. 
In a sense, no genuinely new idempotents appear when considering half-intervals. Beyond the identity object $\oneb$ we only need the objects $X_0^+, X_0^-$ coming from $\mcK$, tensor products of these objects and summands of these tensor products to describe a category equivalent to $\Kar(\mcC_{\alpha})$. 

\vspace{0.1in}

{\it General one-variable case.}  
Let us specialize again to a single variable to generalize examples~\ref{ex_2}--\ref{ex_3}. Consider a general one-variable case $\Sigma=\{a\}$, with rational evaluation functions as in \eqref{eq_generating}, \eqref{eq_encode}. State space $A(+)$ depends on $Z_{\I}(T)$ only, and let 
\begin{equation}
    Z_\I(T) = \frac{P_{\I}(T)}{Q_{\I}(T)}, \ \ \ n_\I=\deg(P_{\I}(T)), \ \ m_\I=\deg(Q_{\I}(T)), 
\end{equation}
for polynomials $P_{\I}(T),Q_{\I}(T)$ with $Q_{\I}(0)\not= 0$. 
Consider the polynomial 
\begin{equation}\label{eq_g_I}
    g_{\I}(T) \ := \ T^{m_\I} Q_{\I}(1/T) \, T^{\max(0,n_\I-m_\I+1)} \ = \ T^{m_\I+\max(0,n_\I-m_\I+1)} Q_{\I}(1/T),
\end{equation}
the product of the reciprocal polynomial of $Q_{\I}(T)$ and a power of $T$. 
Then $g_{\I}(T)$ is the characteristic polynomial for the action of $a$ on $A(+)$, see~\cite{Kh2,KKO} for the equivalent case of the state space of a circle in a two-dimensional topological theory. In particular, 
\begin{equation}
    \dim A(+) \ = \ \dim A(-) \  = \ \deg(g_{\I}(T)) \ = \ m_\I + \max(0,n_\I-m_\I+1). 
\end{equation}
and $A(+)$ is a cyclic $\kk[a]$-module given by 
\begin{equation}
A(+) \ \cong \ \kk[a]/(g_{\I}(a)). 
\end{equation} 

Note that for any rational series 
\begin{equation}
    Z(T) = \frac{P(T)}{Q(T)}, \ \ \ n=\deg(P(T)), \ \ m=\deg(Q(T)), \ \ Q(0)\not= 0, 
\end{equation}
we can form the corresponding polynomial 
\begin{equation}\label{eq_g_Z}
    g_{Z}(T) \ := \ T^{m} Q(1/T) \, T^{\max(0,n-m+1)} \ = \ T^{m+\max(0,n-m+1)} Q(1/T),
\end{equation}
which is the characteristic polynomial for the action of $a$ on the state space $A(+)$ associated to the series $Z(T)$. 

\vspace{0.1in} 

In the direct product decomposition $A(+-)\cong \mcI\times \mcK$ the factor 
$\mcI\cong A(+)\otimes A(-)$ can be understood from this data. To handle $\mcK$, let us start with the potentially larger algebra $U$. These algebras fit into a diagram of surjective homomorphisms 
\begin{equation}\label{eq_many_p} 
     \kk[a] \ \stackrel{p_0}{\lra} \ U \ \stackrel{p}{\lra} \ \mcK, \ \ \  p_{\ast}: \kk[a]\lra \mcK, \ \ p_{\ast}:=p\, p_0,
\end{equation}
where map $p_0$ sends $a^n$ to an arc with $n$ dots (see Figure~\ref{linear-0007} on the right and specialize $\omega=a^n$), and $p$ is the composition of the inclusion of $U$ into $A(+-)$ and projection onto $\mcK$, see \eqref{eq_p}. Finally, let $p_{\ast} := p\circ p_0$, see \eqref{eq_p_sigma}. 

Algebra $U$ is naturally a quotient $U\cong \kk[a]/(g(a))$, for some polynomial $g(a)$. A polynomial $f(a)$ is $0$ in $U$ if and only if it (1) acts trivially on $A(+)$ (equivalently, $f(a)\in (g_{\I}(a)$)) and (2) belongs to the kernel of the bilinear form on $U$ given by $(f_1,f_2)=\tr(f_1f_2)$, with the trace given by closing an arc into into a circle and evaluating it via $\alphac$, see the top right diagram in Figure~\ref{linear-0024}. 

Evaluation $\alphac$ is encoded by a rational series 
\begin{equation}
    Z_\circ(T) = \frac{P_\circ(T)}{Q_\circ(T)}, \ \ \ n_\circ=\deg(P_{\I}(T)), \ \ m_\circ=\deg(Q_{\I}(T)), 
\end{equation}
with $Q_\circ(0)\not= 0$. Form the polynomial (as in \eqref{eq_g_I})
\begin{equation}\label{eq_g_circ} 
    g_{\circ}(T) \ := \ T^{m_\circ} Q_\circ(1/T) \, T^{\max(0,n_\circ-m_\circ+1)} \ = \ T^{m_\circ+\max(0,n_\circ-m_\circ+1)} Q_{\circ}(1/T),
\end{equation}
Then the principal ideal $(g_{\circ}(a))\subset \kk[a]$ is the kernel of the above bilinear form, and $U$ is naturally the quotient of $\kk[a]$ by the intersection of the two principal ideals $(g_\I(a))$ and $(g_\circ(a))$. Let $g_{\alpha}(a) := \mathsf{lcm}(g_\I(a),g_{\circ}(a))$ be the lcm of these two polynomials. We get a canonical isomorphism 
\begin{equation} 
 U \ \cong \  \kk[a]/(g_{\alpha}(a)) \ \cong \ \kk[a]/(\mathsf{lcm}(g_\I(a),g_{\circ}(a))). 
\end{equation} 

\vspace{0.1in}

Coming back to $\mcK$,  recall that it is a quotient of $U$ and $\kk[a]$, see \eqref{eq_many_p} and \eqref{eq_p_sigma}, via quotient maps $p$ and $p\, p_0$, respectively.  

Form the trace series $Z_{\I}^{\tr}:= Z_{\alphai}^{\tr}(T)$ of $Z_{\alphai}(T)$, see formulas \eqref{eq_dec_part}-\eqref{eq_Z_mu}, and the difference of the two series 
\begin{equation}
     Z_{\circ\I}(T) \ := \ Z_{\circ}(T) - Z_{\I}^{\tr}(T).
\end{equation}
Let $g_{\circ\I}(T)$ be the characteristic polynomial for the action of $a$ associated to these series, see \eqref{eq_g_Z}. 

The relation \eqref{eq_tr_pstar} on the trace, reproduced below, 
\begin{equation}\label{eq_tr_pstar_again}
     \tr (p_{\ast}(\omega)) \ = \ \alphac(\omega) - \alphai^{\tr}(\omega),  \ \  \omega\in \Sigma^{\ast}, 
\end{equation}
implies that under the homomorphism $p_{\ast}$ in \eqref{eq_many_p} the trace on $\omega=a^n$, viewed as an element of $\mcK$, is given by the coefficient at $T^n$ of the series $Z_{\circ\I}(T)$ above. Consequently, there is a natural isomorphism 
\begin{equation}
    \mcK \ \cong \ \kk[a]/(g_{\circ\I}(a)), 
\end{equation}
with $g_{\circ\I}(T)$ associated to the series $Z_{\circ\I}(T)$. 
The quotient map $U\lra \mcK$ takes $a\in U$ to $p_{\ast}(a)\in \mcK$. From the isomorphisms 
\begin{equation}
    U\ \cong \ \kk[a]/(\mathsf{lcm}(g_\I(a),g_{\circ}(a))), \ \ \mcK \ \cong \ \kk[a]/(g_{\circ\I}(a))
\end{equation}
that respect the surjection, we obtain an inclusion of ideals $(\mathsf{lcm}(g_\I(a),g_{\circ}(a)))\subset (g_{\circ\I}(a))$, so that the one-variable polynomial $g_{\circ\I}(T)$ is a divisor of $\mathsf{lcm}(g_\I(T),g_{\circ}(T))$.

%
%

\section{From 1D to 2D} 
\label{section_2d}


\subsection{Category of thin flat surfaces and symmetric Frobenius algebras}
\label{subsec_thin_flat} 

Factorization $A(+-)\cong \mcI\times \mcK$ allows us to understand, as described above, the category $\Kar(\mcC_{\alpha})$. The term $\mcI$ contributes objects isomorphic to $\oneb$ to the Karoubi envelope, while the symmetric  Frobenius algebra $\mcK$ leads to a $\mcK$-decorated Frobenius--Brauer category and its negligible quotient, which is equivalent to  $\Kar(\mcC_{\alpha})$. 

Thus, from a one-dimensional theory with defects we can extract a key piece of structure, the symmetric Frobenius algebra $(\mcK,\tr_{\mcK})$ which, on one hand, describes $\Kar(\mcC_{\alpha})$, and on the other has a strongly two-dimensional flavor, since it gives a tensor functor on a suitable category of 2-dimensional cobordisms with corners. 

In this section we review the two-dimensional nature of non-necessarily commutative symmetric algebras $(B,\tr_B)$ and their interpretation via tensor functors from the category of \emph{thin flat surfaces with boundary and corners} to $\kkvect$. A version of this correspondence goes back at least to Moore-Segal~\cite{MS06}, also see~\cite{Lauda05,lauda2006open,LP07,LP08,Cap13,KQR}  and unpublished notes~\cite{KQ-notes-20}. Our modest contribution is to point out that elements of $B$ can be placed as labelled dots ($0$-dimensional defects) along the inner edges of cobordisms and to provide a neck-cutting formula in this language. We discuss obstacles in the nonsemisimple case to extending a thin surface TQFT to an open-closed TQFT as in~\cite{MS06,LP07,LP08,Cap13}. We also show how a  combination of a symmetric Frobenius algebra and an evaluation series for closed cobordisms leads to a universal construction which occupies an intermediate role between open-closed two-dimensional TQFT~\cite{MS06,LP07,LP08} and general two-dimensional topological theories for surfaces with corners~\cite{KQR}. 
\vspace{0.1in}

The category $\TCobt$ of \emph{thin flat surfaces} is defined in~\cite{KQR}, where it is denoted $\mathrm{TFS}$. It has objects $n\in \Z_+=\{0,1,2,\dots \}$ represented by $n$ disjoint intervals in $\R$ ordered from left to right. Morphisms from $n$ to $m$ in $\TCobt$ are surfaces $S$ with boundary and corners together with an immersion into the strip $\R\times [0,1]$, see Figure~\ref{A1} for an example. A surface inherits an orientation from that of $\R\times [0,1]$. An immersion can have overlaps, that can be perturbed in $\R^2\times [0,1]$ into an embedding of $S$ into the latter. Two morphisms are equal if the corresponding oriented surfaces are diffeomorphic rel boundary (forgetting the immersion or the embedding). We refer to~\cite{KQR} for details. 

\input{A1}

Category $\TCobt$ is symmetric monoidal, with the monoidal structure given by placing objects and morphisms next to each other, in parallel. Figure~\ref{A2} shows a possible set of generating morphisms of the monoidal category $\TCobt$. 
Figure~\ref{A3} shows some relations in $\TCobt$, and we refer to~\cite{Lauda05,KQR} for a complete set of relations. 

\input{A2}

\input{A3}

\vspace{0.1in}

The following result is well-known, {\it c.f.},~\cite{MS06,LP07,LP08,LP09,KQ-notes-20}.

\begin{prop} \label{prop_symmetric}
Symmetric monoidal functors $F: \TCobt\lra \kkvect$ are classified by symmetric Frobenius $\kk$-algebras $(B,\tr_B)$, assigning multiplication, unit, trace, comultiplication and permutation maps to cobordisms in Figure~\ref{A1}. 
\end{prop} 
Denote the functor associated to $(B,\tr_B)$ by $F_B$. In this notation we suppress the dependence of $F_B$ on the trace map $\tr_B$. Note that the symmetric trace condition comes from the Figure~\ref{A2.X} equality of morphisms in $\TCobt$. 

\input{A2.X}

\vspace{0.1in} 

The more familiar correspondence is that between \emph{commutative} Frobenius algebras and symmetric monoidal functors from the category $\Cob_2$ of oriented surfaces with boundary to $\kk-vect$. In that case object $n$ of $\Cob_2$ is represented by $n$ circles, not intervals, and the Frobenius algebra is commutative. 

In this case of cobordisms between closed 1-manifolds and commutative Frobenius algebras $(A,\tr_A)$ the functor can be enhanced by introducing dots floating on the components of a cobordism and labelled by elements of $A$. Such dotted cobordism can be evaluated to a linear map between tensor powers of $A$, with dot $a$ denoting multiplication by $a$ map $m_a:A\lra A, m_a(x)=ax$. 

\vspace{0.1in} 

A similar enhancement exists for tensor functors $F_B:\TCobt\lra \kkvect$ as above and we could not find it in the literature. We explain it now.  

A dot labelled $a\in B$ on a side boundary denotes the endomorphism of multiplication by $a$ in $B$. The endomorphism is the left multiplication $\ell_a$ by $a$ if the local orientation at the dot is up and right multiplication by $a$ if the local orientation at the dot is down, see Figure~\ref{Y1}. 

\vspace{0.1in} 

\input{Y1}

Likewise, an $a$-labelled dot at the side boundary of a half-disk denotes either the element $a\in B$ or the map $x\mapsto \tr(ax)$ from $B$ to $\kk$, see Figure~\ref{Y1}. 

Dots can move freely along side boundaries and relations in Figure~\ref{Y2} hold. 
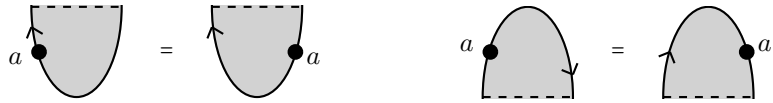
Figure~\ref{Y2.1} depicts that dots may slide past local maxima and minima.

\input{Y2}
 
\input{Y2.1}

\input{Y3}

\vspace{0.1in}

Pick a basis $\{x_i\}_{i=1}^n$ of the vector space $B$ and the dual basis $\{y_i\}_{i=1}^n$ of $B$ with respect to the trace form, so that $\tr(x_iy_j)=\delta_{i,j}$. 
Then 
the surgery relation holds, see Figure~\ref{Y3} left, that allows to cut a surface along an interval connecting two side boundary points. In particular, any 
closed surface in $\TCobt$ (a surface with empty horizontal boundary, that is, an endomorphism of object $\emptyset_1$), with side boundary decorated by elements of $B$, can be evaluated to an element of $\kk$. Indeed, each connected component of such surface has nonempty side boundary and the surgery relation can be iteratively applied for an evaluation. 

One can further augment possible decorations by allowing elements $c$ of the center $Z(B)$ to float inside a surface. Such a floating dot denotes the multiplication by $c$ endomorphism of $B$. A floating central element $c$ can land onto a side boundary, see Figure~\ref{Y4}. 

\input{Y4}

\input{Y5}

\vspace{0.1in} 

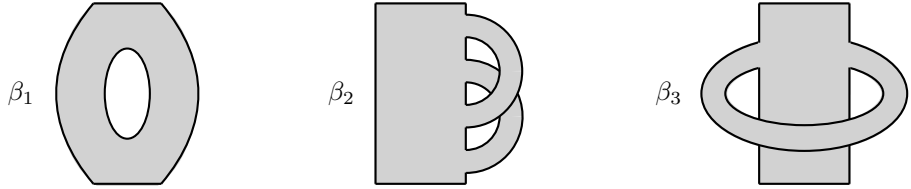
Recall from~\cite{KQR} that category $\TCobt$ has three commuting endomorphisms $\beta_1,\beta_2,\beta_3$ of object $1$, see Figure~\ref{Y5}. These endomorphisms were denoted $b_1,b_2,b_3$ in~\cite{KQR} and  satisfy the relation $\beta_1\beta_3=\beta_3^2$. Consider the functor $F_B$ from $\TCobt$ to $\kkvect$. It takes these endomorphisms to $\kk$-linear maps $F_B(\beta_1),F_B(\beta_2),F_B(\beta_3)$. 

\begin{prop}\label{prop_bimod_fact}
  $F_B(\beta_1)$ and $F_B(\beta_2)$ are $B$-bimodule endomorphisms $B\lra B$, determined by central elements $b_1:=F_B(\beta_1)(1)$ and $b_2:=F_B(\beta_2)(2)$, $b_1,b_2\in Z(B)$. Map $F_B(\beta_3)$ can be factored as 
  \begin{equation}\label{eq_factor} 
      B \stackrel{q_B}{\lra} [B,B] \stackrel{\beta_B}{\lra} Z(B) \stackrel{\iota_B}{\lra} B
  \end{equation}
  for a unique linear map $\beta_B$, 
  where $q_B$ is the quotient map by the commutator subspace and $\iota_B$ the inclusion of the center into $B$.
\end{prop}

\begin{proof}
Maps $F_B(\beta_1)$ and $F_B(\beta_2)$ are $B$-bimodule maps, since $a$-dots can be slid along the side edges of these cobordisms from near a corner at the bottom to the matching corner point at the top, see Figure~\ref{Y6}.

\input{Y6}

\input{Y7}

\vspace{0.1in}

For the map $F_B(\beta_3)$, a dot $a$ near the bottom left corner can be moved to the dot $a$ near the bottom right corner, by dragging it along a side boundary, see Figure~\ref{Y7} on the left. Consequently, $F_B(\beta_3)(ab)=F_B(\beta_3)(ba)$ for any $a,b\in B$, so that $F_B(\beta_3)$ factors through map $q_B$. 

Likewise, $a$-dot at the top left corner of $\beta_3$ can be moved to the top right corner of $\beta_3$, by dragging it along a side boundary, see Figure~\ref{Y7} on the right. This means 
$a F_B(\beta_3)(b)=F_B(\beta_3)(b)a$ for any $a,b\in B$, that is, 
$F_B(\beta_3)(b)$ belongs to the center $Z(B)$ of $B$. Hence, $F_B(\beta_3)$ admits a factorization \eqref{eq_factor}. 
\end{proof}

The map 
\begin{equation}\label{eq_map_B}
    \beta_B \ : \ B/[B,B]\ \lra \ Z(B)
\end{equation}
is associated to a Frobenius algebra $B$ with a symmetric trace. 

\begin{example}\label{ex_map_0} 
  Consider the group algebra $B=\kk[C_p]$  of a prime order cyclic group $C_p=\{1,g|\, g^p=1\}$ over a field $\kk$ of characteristic $p$ with the trace
  \begin{equation}
  \tr_B:B\lra \kk, \  \ \tr_B(1)=1, \ \ \tr_B(g^i)=0, \ 0<i<p.
  \end{equation}
  The comultiplication structure map of this thin surface TQFT is 
  \begin{equation}
      \Delta(g^i)=\sum_{i=0}^{p-1} g^j\otimes g^{i-j}. 
  \end{equation}
  The map $F_B(\beta_3)$ is the composition 
  \begin{equation}
       (\tr_B\otimes 1)(m\otimes 1)(1\otimes P) (1\otimes \Delta)\Delta,
  \end{equation}
  see Figure~\ref{Y8}. 
Computing this map for the group algebra $B$ yields $F_B(\beta_3)=0$ and $\beta_B=0$.

\input{Y8}
\end{example}

Choosing a finite set $\Sigma$ of labels for the side boundary dots, one obtains a decorated version $\TCob_{2,\Sigma}$ of the thin flat surface category, where side boundaries can carry dots labelled by elements of $\Sigma$ as in Figures~\ref{Y1},~\ref{Y2.1}, for instance. A symmetric Frobenius algebra $(B,\tr_B)$ as above together with a map of sets $\Sigma\lra B$ determines a symmetric monoidal functor $\TCob_{2,\Sigma}\lra \kkvect$. 

\begin{example} \label{eq_continue_ex} For Frobenius $B$ in Example~\ref{ex_map_0} and $\Sigma=\{g\}$, the relations for the resulting functor are shown in Figure~\ref{L1}, where a dot labelled $i\in \Z/p$ denotes multiplication by $g^i$ at that position (a dot labelled $0$ can be erased). 

\vspace{0.1in} 

\input{L1}

\end{example}


\subsection{Open-closed TQFTs and topological theories} 
\label{subsec_open_closed} 

\quad 

{\it Open-closed 2D TQFTs and knowledgeable Frobenius algebras.}
The usual category of oriented cobordisms between closed one-manifolds and the category of thin flat surfaces $\TCobt$ can be unified into the category of open-closed 2D cobordisms~\cite{MS06,LP08,Laz01,lauda2006open}. In that category $\OCCobt$ objects are compact oriented 1-manifolds with boundary, thus finite unions of intervals and circles. Morphisms are diffeomorphism classes (rel boundary) of oriented surfaces with boundary and corners. These surfaces have \emph{side boundary} and corner points where side boundaries meet top and bottom boundary intervals. We refer to~\cite{LP08} for details. As a set of generating morphisms one can take the union of 
\begin{itemize}
\item standard generating morphisms for thin surfaces in Figure~\ref{A2} that correspond to the structure maps of a symmetric Frobenius algebra, and the permutation map $P$ of two intervals, 
\item standard generating morphisms for the usual category of two-dimensional oriented cobordisms. These morphisms correspond to the structure maps in a commutative Frobenius algebra and the transposition of two circles, 
\item Zipper and cozipper cobordisms, see Figure~\ref{open-closed}, and the transposition of an interval and an circle, similar to the transposition $P$ of two intervals in Figure~\ref{A2}. 
\end{itemize}

\input{open-closed}

A 2D TQFT for the open-closed category consists of a symmetric Frobenius algebra $(B,\tr_B)$, which is the state space of an interval, and a commutative Frobenius algebra $(C,\tr_C)$, describing the state space of a circle. These two Frobenius algebras are subject to the following interactions: 
\begin{enumerate}
    \item There exists a trace-respecting coalgebra homomorphism $\jmath: B\lra C$ (the \emph{zipper homomorphism}),  see Figure~\ref{open-closed} left. Its dual is 
 the \emph{cozipper} map $\jmath^{\ast}:C\lra B$, see Figure~\ref{open-closed} right. That $\jmath$ is a coalgebra 
 homomorphism comes from diffeomorphisms (rel boundary) in the 
 top row of Figure~\ref{zipper-001}; $\jmath$ intertwines comultiplication and the counit maps of the two coalgebras.  
 The dual (cozipper) map intertwines the two multiplications and takes the unit element of $C$ to the unit element of $B$, thus it is an algebra homomorphism. These properties of $\jmath^{\ast}$ correspond to the 
 relations given by rotating those in the top row of Figure~\ref{zipper-001} by $180^{\circ}$.
    \item Maps $\jmath,\jmath^{\ast}$ are subject to
    \begin{itemize}
        \item The 
    \emph{knowledge} relation, shown on the left in the second row in~\ref{zipper-001},
    \item The duality between the zipper and the cozipper, shown 
    on the right in the second row of~\ref{zipper-001}, 
    \item 
    The Cardy condition, see the third row in~\ref{zipper-001}.
    \end{itemize}
\end{enumerate}

\begin{remark}
Lauda and Pfeiffer~\cite{LP08,LP09} view a cobordism as a morphism from the top boundary to the bottom boundary, while our convention is the opposite. They also denote a symmetric Frobenius algebra by $A$ instead of our $B$, and zipper morphism by $\iota$ instead of $\jmath$, so that, for instance, the zipper morphism is $C\stackrel{\iota}{\lra}A$ in~\cite{LP08,LP09} and $B\stackrel{\jmath}{\lra}C$ in the present paper. 
\end{remark}

\input{zipper-001}


Such pairs $(B,C)$ are called \emph{knowledgeable Frobenius algebras}, see~\cite[Definition 2.2]{LP09}. Examples of many such pairs in the literature are build from a \emph{strongly separable} symmetric Frobenius algebra $B$. An algebra is called strongly separable if 
the trace form $(a,b)_{\ell}=\tr_A(L_a\circ L_b)$ is non-degenerate, where $L_a$ is the operator $A\lra A$ of left multiplication by $a$. 

A strongly separable symmetric Frobenius algebra $B$ extends to a knowledgeable Frobenius algebra $(B,C)$ by taking $C=Z(A)$ to be the center of $A$, see~\cite{LP07,LP09}.   
A strongly separable algebra is necessarily semisimple. 

\begin{example}
Consider the thin surface TQFT and the Frobenius algebra $B$ in Examples~\ref{ex_map_0},~\ref{eq_continue_ex}. 
One can look to extend this open TQFT $F$ to an open-closed TQFT (to a knowledgeable Frobenius algebra $(B,C)$). Then the  commutative Frobenius algebra $C$ associated to a circle should have a  distinguished element $x\in C$, $x=\jmath(1_B)$ shown in Figure~\ref{L2} in the top row and given by a cup with a hole in it. It is the image of the identity element $1_B$ of $B$ under the zipper map $\jmath$. The trace $\tr_C(x)=1$ since a dotless disk evaluates to $1$, so that $x\not= 0\in C$. On the other hand, $x^2=0$ since the corresponding cobordism contains a subsurface shown in the middle of the second row in Figure~\ref{L1}, which evaluates to $0$. Thus, $C$ contains a subalgebra $\kk[x]/(x^2)$. 

\vspace{0.1in} 

\input{L2}

\vspace{0.1in} 

More generally, $C$ contains elements $x_i=\jmath(g^i)$, $i\in \Z/p$, with $x=x_0$, see Figure~\ref{L2} bottom row. We have 
\begin{equation}\label{eq_xs} 
    x_i \, x_j = 0, \  i,j\in\Z/p, \ \  \  \Delta_C(x_i)=\sum_{k=0}^{p-1} x_{i+k}\otimes x_{p-k}, \ \ \tr_C(x_i) = \delta_{i,0}, \ \  \jmath^{\ast}(x_i)=0, \ i\in \Z/p. 
\end{equation}
These elements (or their linear combinations) may potentially be $0$ for $i\not=0$. The trace on $C$ must be nondegenerate, pairwise products of $x_i$'s are $0$ and $\tr_C(x_i)=0$ for $i\not=0$. To avoid introducing more generators for $C$, we further assume that $x_i=0$ for $i\not=0$ and look to complete $B$ to a knowledgeable Frobenius algebra $(B,C)$ with $C\cong \kk[x]/(x^2)$.  Then the zipper and the cozipper maps are 
\begin{eqnarray*}
    \jmath: B\lra C,  & & \jmath(g^i)=0, \ i\not= 0\in \Z/p, \ \ \jmath(1_B)= x, \ \ \mathsf{im}(\jmath)=\kk \, x, \\ 
    \jmath^{\ast}: C \lra B, & & \jmath^{\ast}(1_C)= 1_B, \ \ \jmath^{\ast}(x) = 0, \ \ \mathsf{im}(\jmath^{\ast}) = \kk \, 1. 
\end{eqnarray*}
with $\jmath^{\ast}\jmath=0$, map $\jmath$ a coalgebra homomorphism and $\jmath^{\ast}$ an algebra homomorphism. Each individual map $\jmath,\jmath^{\ast}$ is nonzero but has a one-dimensional image. 
To define the trace on $C$, pick a parameter $\lambda\in\kk$: 
\begin{equation}
    \tr_C(1) \ = \ \lambda, \ \ \ \tr_C(x) = 1. 
\end{equation}
The multiplication ($x^2=0$) and trace on $C$ determine the comultiplication 
\begin{equation}
    \Delta_C(1) =  1\otimes x + x\otimes 1 -\lambda x\otimes x, \ \
    \Delta_C(x) =  x \otimes x. 
\end{equation}
The remaining open-closed TQFT relations, as shown in Figure~\ref{zipper-001}, are straightforward to check. 

In this knowledgeable Frobenius algebra, both $B$ and $C$ are nonsemisimple algebras. The handle (or punctured torus) element of $C$ equals $m_C\circ \Delta_C(1)=2x$, and a closed surface of genus $g$ evaluates to $\alpha_{0,g}$, where  
\begin{equation}
    \alpha_{0,0}=\lambda, \ \ \  \alpha_{0,1}= 2, \ \ \ \alpha_{0,g}=0 \ \mathrm{for} \ g>1 
\end{equation}
(recall that the coefficients belong to a field $\kk$ of characteristic $p$). 
Evaluation $\alpha_{m,g}$ of a connected surface of genus $g$ with $m$ side boundary circles is 
\begin{equation}
    \alpha_{1,0} = 1, \ \ \ \alpha_{m,g}=0 \ \mathrm{if} \ m>0 \ \mathrm{and} \ g\not=0.
\end{equation}
Thus, in this evaluation, at most three coefficients: $\alpha_{0,0},\alpha_{0,1},\alpha_{1,0}$ are nonzero. 

More generally, $\jmath^{\ast}\jmath=0$ if and only if $\alpha_{m,g}=0$ for all $m\ge 1,g\ge 1$ and $m\ge 2,g\ge 0$.

Example 3.7 in~\cite{LP09} is somewhat similar to the present example, with the algebra $B$ isomorphic to the one above (and $\mathsf{char}(\kk)=p$), algebra $C=\kk[x]/(x^2-ht-t)$ two-dimensional and nonsemisimple when parameters $h,t$ satisfy $h^2=4t$, but with  different zipper and trace maps. In particular, $\jmath^{\ast}\jmath\not=0$ in that example ($\iota\,\iota^{\ast}\not=0$ in the notations of~\cite{LP09}). 
\end{example}

In the above example of a knowledgeable Frobenius pair $(B,C)$ both $B$ and $C$ are nonsemisimple and the maps $\beta_B$ and $F_B(\beta_3)$, see \eqref{eq_factor} and \eqref{eq_map_B}, are zero, so that $\jmath^{\ast}\jmath=0$. 
    
\vspace{0.1in}   
  
Frobenius algebras that appear in link homology and categorification are typically nonsemisimple, which creates an obstacle to merging link homology with open-closed 2D TQFTs, where a vast majority of examples is built from semisimple Frobenius algebras. This obstruction is  discussed in~\cite{LP09} and~\cite{Cap13}. One well-known way out of this is a functorial extension of link homology to tangles~\cite{Kho-functor-tangles-02} and then to tangle cobordisms~\cite{Kho-tangle-06,BN05}.

This discrepancy between semisimple Frobenius algebras common in open-closed TQFTs in dimension two and rather special nonsemisimple Frobenius algebras that give rise to link cobordism TQFTs in dimension four and categorification of quantum invariants is an interesting phenomenon that is not fully understood.  

\vspace{0.1in} 

{\it Combining a symmetric Frobenius algebra with the universal construction for closed surfaces.}
Even when symmetric Frobenius $B$ does not extend to a knowledgeable Frobenius $(B,C)$ it is possible to extend $B$ to a functor from $\OCCobt$ to the category of vector spaces but with a weaker axioms than that of a TQFT. This can be achieved by combining the Frobenius structure $(B,\tr_B)$ with the universal construction. 

Symmetric Frobenius algebra $(B,\tr_B)$ gives a thin flat surface TQFT. In particular, it evaluates any connected oriented surface $S_{n,g}$ with $n\ge 1$ boundary components and $g$ handles to a number $\alpha_{B,n,g}\in \kk$. This number can be computed by viewing the surface as an endomorphism of the identity object $0$ of $\TCobt$ and computing the element of $\kk$ it goes to under the functor $F_B$. Alternatively, one can use the surgery relation in Figure~\ref{Y3} and other relations in Figures~\ref{Y3},~\ref{Y2}. 

Doing the universal construction in the modification of $\TCobt$ where side boundaries are decorated by generators of $B$ results in the TQFT $F_B$. 

To extend to all oriented surfaces, we choose an evaluation $\alpha_{0,g}\in \kk$ of a closed oriented surface  of genus $g$ for all $g\ge 0$. It is convenient 
to require that the generating function 
\begin{equation}\label{eq_Z_0}
    Z_0(T)\ :=\ \sum_{g\ge 0} \alpha_{0,g} T^g
\end{equation}
is rational. With this additional choice evaluations of all surfaces $S_{n,g}$, $n,g\ge 0$ of genus $g$ with $n$ boundary circles are defined.

The universal construction can be applied to the category $\OCCobt$ of open-closed cobordisms. To get a better match with state spaces build from $(B,\tr_B)$ for cobordisms that have corners, it is convenient to pick a set $W$ of generators of $B$ and allow these generators to float on side boundaries of cobordisms. 
One obtains a minor modification, denoted $\OCCobt(W)$ of the category $\OCCobt$. Endomorphisms of the $0$ object of $\OCCobt$ that come from surfaces with decorated side boundaries are then evaluated via $(B,\tr_B)$ while evaluations of closed surfaces are encoded in the generating function \eqref{eq_Z_0}. 

Let us do the universal construction for $\OCCobt(W)$ evaluating surfaces with side boundary their possible $W$-decorations via $(B,\tr_B)$ and closed surfaces via coefficients of \eqref{eq_Z_0}. The resulting category and a functor is an extension of the thin surface TQFT associated with $(B,\tr_B)$. Objects of $\OCCobt(W)$ are finite unions $I^{\sqcup k}\sqcup (\S^1)^{\sqcup m}$ of intervals and circles. Denote the state space of that one-manifold by $A(k,m)$. Then the surgery formula still applies near each  interval component, and the state space simplifies via the isomorphism 
\begin{equation}
     A(k,m) \ \cong \ B^{\otimes k} \otimes A(0,m). 
\end{equation}
That is, the state space is isomorphic to the tensor product of $B$'s, one for each interval, and the state space of $m$ circles. The latter state space contains a subquotient isomorphic to the state space of $m$ circles in the closed 2D topological theory with the generating function \eqref{eq_Z_0}, as studied in~\cite{Kh2,KS3,KKO}. The state space $A(0,m)$ may be strictly bigger than the latter state space, due to the presence of surfaces that bound $m$ circles at the top but have side circles (such surfaces can be viewed as morphisms in $\OCCobt(W)$ from the identity object $0$ to $m$ circles). 

\vspace{0.1in} 

This universal construction based on $(B,\tr_B)$ and rational power series \eqref{eq_Z_0} occupies an intermediate position between open-closed TQFTs and the more general universal construction for surfaces with boundary and corners studied in~\cite{KQR}. In the present case $(B,\tr_B)$ allows to evaluate surfaces of all genera with at least one side boundary circle and produce a TQFT (as long as we add additional observables on the boundary lines for generators of $B$) for these surfaces, then extend to cobordisms that may have top and bottom boundary circles and closed surfaces via \eqref{eq_Z_0}. Without enlarging $\OCCobt$ to $\OCCobt(W)$ the resulting state space of the union of $k$ intervals could be only a subspace of $B^{\otimes k}$.  

\begin{remark} Going back to the category $\TCobt$ of thin flat surfaces, one can further introduce one-dimensional interval defects that connect two points on the boundary of a surface. These defects are labelled by endomorphisms of $B$ (by $\kk$-linear maps $B\to B$), see Figure~\ref{Z1}.

\input{Z1}

\end{remark}

\begin{remark}
Categorifications of the Heisenberg algebra come from the study of natural transformations on compositions of the  induction and restriction functors between symmetric groups or Hecke algebras~\cite{Kho14,LS13}.
More general categorifications of the Heisenberg algebra~\cite{BS22,Sav19,RS17,BSW21}  
add elements of a Frobenius algebra as decorations on strands of diagrams in those graphical calculi. The neck-cutting formula  in Figure~\ref{Y3} on the left is called \emph{the Frobenius skein relation} in that case and is referred to as \emph{teleportation} in \cite{BS22}.

A symmetric Frobenius algebra gives a TQFT for thin flat surfaces (two-dimensional objects), which is one of the indications that various Heisenberg algebra categorifications should admit reformulations via a suitable graphical calculus of foam-like objects in $\R^3$ rather than graphs (or intersecting decorated lines) in $\R^2$. 
\end{remark} 

%
%

\section{Embeddings into a 1D TQFT and dimensional lifting} 

{\it Embeddings into a 1D TQFT: semisimplicity restriction.}
It is natural to ask under what conditions on $\mcK$ and the trace $\tr_{\mcK}$ is the corresponding one-dimensional theory of arcs with $\mcK$-defects and the circle evaluation given by $\tr_{\mcK}$ embeddable into a one-dimensional TQFT. Oriented 1D TQFTs are described by finite-dimensional vector spaces $V$, with the state space of $+-$ oriented 0-manifold isomorphic to $V\otimes V^{\ast}$. When viewed as an algebra under the composition in Figure~\ref{linear-0023} it is naturally isomorphic to the endomorphism or the matrix algebra $\End_{\kk}(V)\cong M_n(\kk)$. 

Suppose given a symmetric Frobenius algebra $(B,\tr_B)$. It gives rise to the category of arcs with $B$-defects and circles evaluated via $\tr_B$ and the negligible quotient of that category. A monoidal functor from either of these two categories into a 1D TQFT given by $V$ is described by 
 a homomorphism of algebras $\phi: B\lra M_n(\kk)$ that converts trace $\tr_B$ to the usual trace on the matrix algebra $M_n(\kk)$, that is 
$\tr(\phi(a)) = \tr_B(a), \forall a\in B$. That is, $\phi$ must intertwine the two traces. Then $\phi$ is necessarily an inclusion and $\tr_B(1)=n=\dim(V)$. 

Consider the Jacobson radical $J\subset B$. Then $J$ is a two-sided nilpotent ideal, $J^n=0$, and $B/J$ is semisimple. Any element $x\in J$ is nilpotent and $\phi(x)$ is a nilpotent matrix, so that $\tr(\phi(x))=0$. Consequently, $\tr_B(x)=0$, for all elements $x$ in the Jacobson radical. Nondegeneracy of $\tr_B$ implies that $J=0$, so that $B$ is semisimple, and we obtain the following result. 

\begin{prop}
A one-dimensional topological theory with $B$-labelled defects associated to $(B,\tr_B)$ can be embedded into a one-dimensional TQFT over $\kkvect$ only if $B$ is semisimple. 
\end{prop}

We see that $\tr_B$ can come from a trace on a matrix algebra in the above way only if $B$ is a semisimple $\kk$-algebra, while in applications (for instance, to link homology) we most often encounter cases when $B$ is not semisimple. 

Furthermore, assuming that $B$ is semisimple, only few traces on $B$ correspond to embeddings into one-dimensional TQFTs. 
Namely, $B\cong \prod_{i=1}^k \Mat_{n_i}(D_i)$ is then isomorphic to the product of matrix algebras over finite-dimensional division
rings $D_i$ over $F$ and representation $V$ has the form $V\cong \oplus_{i=1}^k V_i$, where 
\begin{equation}
V_i \cong  (D_i^{n_i})^{r_i}
\end{equation}
is the sum of $r_i$ copies of the column representation $D_i^{n_i}$ of the matrix algebra $\Mat_{n_i}(D_i)$. Under this isomorphism, 
\begin{equation}
    \tr_B  \ = \ \sum_{i=1}^k \, r_i \, \tr_i , 
\end{equation}
where $\tr_i$ is the trace on $\Mat_{n_i}(D_i)$ with values in $\kk$ which is the composition of the matrix algebra trace and the map $\tr'_i: D_i \lra \kk$ which is the trace of left multiplication in $D_i$ viewed as a $\kk$-vector space. 

For example, suppose that the division ring $D_i$ is commutative, thus it is a field $F$ such that $\kk\subset F$ is a finite extension and that $n_i=1$. Any non-zero $\kk$-linear map $\varepsilon: F\lra \kk$ turns $F$ into a commutative Frobenius $\kk$-algebra, but only the trace map $\tr_{\kk}:F\lra\kk$ and its multiples 
$r\,\tr_{\kk}$, $r\in \N$ (further assuming that $F/\kk$ is separable) come from embeddings into a 1D TQFT. 

\vspace{0.1in} 

{\it Dimensional liftings.} 
We see that trace-preserving embeddings of symmetric Frobenius algebras into matrix algebras are scarce. 
At the same time, a symmetric Frobenius algebra $(B,\tr_B)$ gives rise to a 2D TQFT for thin surfaces as explained earlier. A one-dimensional TQFT with defects $\alpha$ produces a two-dimensional TQFT, restricted to thin surfaces, via the symmetric Frobenius algebra $(\mcK,\tr_{\mcK})$. This dimensional lifting from one to two dimension can be very loosely compared to the Drinfeld center of a monoidal category (and the Drinfeld double of a Hopf algebra). Monoidal categories are naturally two-dimensional structures, with morphisms often represented by planar diagram. The Drinfeld center of a monoidal category is a braided monoidal category, providing invariants of braids and lifting the structure one dimension up, from two to three dimensions. Likewise, the Drinfeld double of a Hopf algebra converts a two-dimensional structure (the category of representations of a Hopf algebra is monoidal) to a three-dimensional structure (a quasitriangular Hopf algebra, with the category of representations being a braided monoidal category). 

The graphical nature of a monoidal category $\mcC$ is that of planar networks of morphisms between tensor products of objects of $\mcC$. Such planar networks can be thought of as defects in the two-dimensional theory of the underlying plane $\R^2$. Drinfeld's center and doubling constructions lift these ``two-dimensional theories with defects'' to three-dimensional theories. 

Of course, the above discussion and comparison of dimensional liftings is highly informal. 


\bibliographystyle{amsalpha}

\bibliography{top-automata}

\end{document}